\documentclass[reqno, final,a4paper]{amsart}
\usepackage[notref,notcite]{showkeys}
\usepackage{color}
\usepackage[colorlinks=true,allcolors=blue
,backref=page
]{hyperref}
\usepackage{amsmath, amssymb, amsthm}
\usepackage{mathrsfs}
\usepackage{mathtools}
\usepackage[noabbrev,capitalize,nameinlink]{cleveref}
\crefname{equation}{}{}
\usepackage{fullpage}
\usepackage[noadjust]{cite}
\usepackage{graphics}
\usepackage{pifont}
\usepackage{tikz}
\usepackage{bbm}
\usepackage[T1]{fontenc}
\usetikzlibrary{arrows.meta}

\usepackage{environ}
\usepackage{framed}
\usepackage{url}
\usepackage[linesnumbered,ruled,vlined]{algorithm2e}
\usepackage[noend]{algpseudocode}
\usepackage[labelfont=bf]{caption}
\usepackage{framed}
\usepackage[framemethod=tikz]{mdframed}
\usepackage{pgfplots}
\usepackage{appendix}
\usepackage{graphicx}
\usepackage[textsize=tiny]{todonotes}
\usepackage{tcolorbox}
\usepackage[shortlabels]{enumitem}
\crefformat{enumi}{#2#1#3}
\crefrangeformat{enumi}{#3#1#4 to~#5#2#6}
\crefmultiformat{enumi}{#2#1#3}
{ and~#2#1#3}{, #2#1#3}{ and~#2#1#3}
\allowdisplaybreaks
\pgfplotsset{compat=1.18}%Added by Lyuben

\let\originalleft\left
\let\originalright\right
\renewcommand{\left}{\mathopen{}\mathclose\bgroup\originalleft}
\renewcommand{\right}{\aftergroup\egroup\originalright}

\crefname{algocf}{Algorithm}{Algorithms}

\crefname{equation}{}{} %remove ``Equation''
\AtBeginEnvironment{appendices}{\crefalias{section}{appendix}} %appendices

\crefname{algocf}{Algorithm}{Algorithms}

\numberwithin{equation}{section}
\newtheorem{theorem}{Theorem}[section]

\newtheorem{lemma}[theorem]{Lemma}
\newtheorem{claim}[theorem]{Claim}
\crefname{claim}{Claim}{Claims}

\newtheorem*{question*}{Question}

\crefname{conjecture}{Conjecture}{Conjectures}
\newtheorem*{theorem*}{Theorem}

\theoremstyle{definition}
\newtheorem{definition}[theorem]{Definition}

\newtheorem*{definition*}{Definition}

\theoremstyle{remark}
\newtheorem{remark}[theorem]{Remark}
\newtheorem*{remark*}{Remark}

\newcommand{\floor}[1]{\left\lfloor #1 \right\rfloor}
\newcommand{\ceil}[1]{\left\lceil #1 \right\rceil}

\DeclareMathOperator{\e}{\mathrm{e}}
\DeclareMathOperator{\vol}{\mathrm{vol}}

\newcommand{\mb}{\mathbb}

\newcommand{\mc}{\mathcal}

\newcommand{\mr}{\mathrm}

\newcommand{\ol}{\overline}
\newcommand{\on}{\operatorname}

\renewcommand{\Pr}{\mb P}

\DeclareMathOperator{\cE}{\mathcal{E}}
\DeclareMathOperator{\cF}{\mathcal{F}}
\DeclareMathOperator{\cG}{\mathcal{G}}

\newcommand*{\claimproofname}{Proof of claim}
\newenvironment{claimproof}[1][\claimproofname]{\begin{proof}[#1]}{\end{proof}}

\setlist{itemsep=0.3em} 

\allowdisplaybreaks

\title{Colouring random Hasse diagrams and box-Delaunay graphs}

\author[Jin]{Zhihan Jin}

\address{Department of Mathematics, ETH, Z\"urich, Switzerland}
\email{zhihan.jin@math.ethz.ch}

\author[Kwan]{Matthew Kwan}

\address{Institute of Science and Technology Austria (ISTA). Am Campus 1, 3400 Klosterneuburg, Austria.}
\email{matthew.kwan@ist.ac.at}

\author[Lichev]{Lyuben Lichev}
\address{Institute of Science and Technology Austria (ISTA). Am Campus 1, 3400 Klosterneuburg, Austria.}
\email{lyuben.lichev@ist.ac.at}

\begin{document}

\begin{abstract}
Fix $d\ge2$ and consider a uniformly random set $P$ of $n$ points in $[0,1]^{d}$. Let $G$ be the Hasse diagram of $P$ (with respect to the coordinatewise partial order), or alternatively let $G$ be the Delaunay graph of $P$ with respect to axis-parallel boxes (where we put an edge between $u,v\in P$ whenever there is an axis-parallel box containing $u,v$ and no other points of $P$).

In each of these two closely related settings, 
we show that the chromatic number of $G$ is typically $(\log n)^{d-1+o(1)}$
and the independence number of $G$ is typically $n/(\log n)^{d-1+o(1)}$.
When $d=2$, we obtain bounds that are sharp up to constant factors:
the chromatic number is typically of order $\log n/\log\log n$ and
the independence number is typically of order $n\log\log n/\log n$.

These results extend and sharpen previous bounds by Chen, Pach, Szegedy
and Tardos. In addition, they provide new bounds on the largest possible
chromatic number (and lowest possible independence number) of a $d$-dimensional
box-Delaunay graph or Hasse diagram, in particular resolving a conjecture
of Tomon.
\end{abstract}

\maketitle

\section{Introduction}
\newcommand{\Pp}{P_{\mr{Po}}}

In this paper, we consider two related types of graphs arising
from quite different contexts. 
First, there is a vast literature on
graphs defined via geometric configurations; one important class of
such graphs (of particular significance in computational geometry, see e.g.\ \cite{H-PS03}) is the class of \emph{Delaunay graphs}.
\begin{definition*}
Let $\mathcal{C}$ be a family of $d$-dimensional convex bodies in
$\mb R^{d}$, and consider a set of points $P\subseteq\mb R^{d}$. The\emph{
Delaunay graph} $D_{\mathcal{C}}(P)$ of $P$ with respect to $\mathcal{C}$
is the graph on the vertex set $P$ with an edge between $p,q\in P$
whenever there is $C\in\mathcal{C}$ with $C\cap P=\{p,q\}$ (i.e.,
we put an edge between $p$ and $q$ if one of our convex bodies contains
$p$ and $q$ but no other point of $P$).
\end{definition*}
Most famously, if $d=2$ and $\mathcal{C}$ is the family of balls
in $\mb R^{2}$, then $D_{\mathcal{C}}(P)$ is the graph of the \emph{Delaunay
triangulation} of $P$, which is in a certain sense the ``dual''
of the Voronoi diagram of $P$. After balls, perhaps the second
most important case is where $\mathcal{C}$ is the set of axis-parallel
boxes (i.e., the family of all sets of the form $[a_{1},b_{1}]\times\dots\times[a_{d},b_{d}]\subseteq \mb R^d$).
In the case $d=2$, such graphs are sometimes called \emph{rectangular
visibility graphs}; in this paper, we use the term \emph{box-Delaunay
graphs} to describe Delaunay graphs with respect to axis-parallel boxes of any dimension.

\medskip On the other hand, \emph{Hasse diagrams} (also known as \emph{covering
graphs}) are an important class of graphs defined in terms of partial
orders.
\begin{definition*}
Let $(P,\preceq)$ be a partially ordered set (poset). The \emph{Hasse
diagram} $H_{\preceq}(P)$ of this partial order is the graph on the
vertex set $P$ with an edge between $p,q\in P$ whenever we have
$p\preceq q$ and there is no $r\in P$ with $p\preceq r\preceq q$
(i.e., we put an edge between $p$ and $q$ if they are comparable
but there is no other element ``between'' $p$ and $q$).
\end{definition*}
It turns out that there is a very close relation between box-Delaunay
graphs and Hasse diagrams. Indeed, a set of points $P\subseteq\mb R^{d}$
naturally gives rise to a poset $(P,\preceq)$, where $p\preceq q$ if every coordinate of $q$ dominates the respective coordinate of $p$
(the posets that arise in
this way are precisely the posets with \emph{Dushnik--Miller dimension}
at most $d$). If we allow ourselves the freedom to reverse the order in some of the $d$ directions,
we can actually associate $2^{d}$
different posets to $P$, and the union of the Hasse diagrams of these
posets is precisely the box-Delaunay graph of $P$. In particular,
every Hasse diagram can be viewed as a subgraph of a box-Delaunay
graph of the same dimension.

\medskip For any class of graphs (especially those defined in terms of geometry
or order), one of the most fundamental questions is how large the
\emph{chromatic number} can be. It was first proved by Bollob\'as~\cite{Bol77}, in 1977, that Hasse diagrams can have arbitrarily large chromatic
number (this result was later independently proved by Ne\v set\v ril
and R\"odl~\cite{NR79}). Also, if an $n$-vertex graph has chromatic
number $\chi$, then it has an independent set of size at least $n/\chi$,
but the converse does not necessarily hold. So, a strengthening of
the fact that Hasse diagrams can have unbounded chromatic number is
that $n$-vertex Hasse diagrams can have independence number $o(n)$;
this was first proved by Brightwell and Ne\v set\v ril~\cite{BN91}.
Quantitative improvements to these results were later obtained by
Pach and Tomon~\cite{PT19} and by Suk and Tomon~\cite{ST21}.

None of the aforementioned results say anything about Hasse diagrams
(or box-Delaunay graphs) of \emph{fixed} dimension, but an ingenious
construction of K\v r\'i\v z and Ne\v set\v ril~\cite{KN91} shows
that even Hasse diagrams of dimension 2 (and therefore box-Delaunay
graphs of dimension 2) can have arbitrarily large chromatic number.
The K\v r\'i\v z--Ne\v set\v ril construction has very poor
quantitative aspects, and says nothing about the independence number;
both these aspects were later addressed in later work by Chen, Pach,
Szegedy and Tardos~\cite{CPST09} (answering a question of Matou\v sek and
P\v r\'iv\v etiv\'y~\cite{MP06}).

In fact, Chen, Pach, Szegedy and Tardos did not just consider worst-case behaviour; they were the first to consider questions of this type for \emph{random} point sets. 
They were able to obtain quite sharp estimates on the typical independence number for random point sets, as follows.
\begin{theorem*}
    Let $P\subseteq[0,1]^{2}$ be a uniformly random set of $n$ points in $[0,1]^{2}$, and let $G$ be either the box-Delaunay graph of $P$ or the Hasse diagram of $P$ (with respect to the natural coordinatewise partial order).
    Then, whp\footnote{We say an event holds \emph{with high probability}, or \emph{whp} for short, if it holds with probability $1-o(1)$ as $n\to\infty$.}, the chromatic number $\chi(G)$ and the independence number $\alpha(G)$ satisfy the inequalities
    \[  
        \frac{c\log n}{(\log\log n)^2}\le \chi(G),\qquad 
        \frac{cn\log \log n}{\log n\log\log\log n}\le \alpha(G)\le\frac{Cn(\log \log n)^2}{\log n},
    \]
for some absolute constants $c,C>0$.
\end{theorem*}
It is worth mentioning that there is an enormous literature on geometrically-defined random graphs (see for example the monograph \cite{Pen03}), but random box-Delaunay graphs have a rather different character than many of the standard notions of random geometric graphs, as they are curiously ``non-local'' (for example, if two points are very close in one coordinate, they are very likely to be connected by an edge, even if they are very far in their other coordinates).

As our first result, we improve both directions of the Chen--Pach--Szegedy--Tardos bounds, obtaining sharp results up to constant factors (and thereby obtaining new bounds on the maximum chromatic number and minimum independence number of an $n$-vertex Hasse diagram or box-Delaunay graph of dimension 2).
\begin{theorem}\label{thm:d=2}
Let $P\subseteq[0,1]^{2}$ be a uniformly random set of $n$ points
in $[0,1]^{2}$, and let $G$ be either the box-Delaunay graph of
$P$ or the Hasse diagram of $P$. Then, whp
\[
\frac{c\log n}{\log\log n}\le\chi(G)\le \frac{C\log n}{\log\log n},\qquad 
\frac{cn\log \log n}{\log n}\le \alpha(G)\le\frac{Cn\log \log n}{\log n},
\]
for some absolute constants $c,C>0$.
\end{theorem}
The Chen--Pach--Szegedy--Tardos proof strategy does not straightforwardly generalise to higher dimensions, but by considering a more complicated probability distribution on box-Delaunay graphs, 
Tomon~\cite{Tom24} was able to prove that, for any constant $d\ge 2$, there exist $n$-vertex box-Delaunay graphs of dimension $d$ which have independence number at most $n/(\log n)^{(d-1)/2+o(1)}$ (and therefore chromatic number at least $(\log n)^{(d-1)/2+o(1)}$). 
Although his proof does not work for Hasse diagrams, he conjectured that a similar bound should 
hold in that setting (namely, he conjectured that there is a constant $c>0$ such that, for any $d$ and sufficiently large $n$, there are $n$-vertex Hasse diagrams of dimension $d$ whose independence number is at least $n/(\log n)^{cd}$).

As our second result, we prove the first results for Hasse diagrams and box-Delaunay graphs of random point sets in dimension $d\ge 3$: the typical chromatic number is $(\log n)^{d-1+o(1)}$ and the typical independence number is $n/(\log n)^{d-1+o(1)}$. In particular, this resolves Tomon's conjecture.

\begin{theorem}\label{thm:d>2}
Fix $d\ge 3$. Let $P\subseteq[0,1]^{d}$ be a uniformly random set of $n$ points
in $[0,1]^{d}$, and let $G$ be either the box-Delaunay graph of
$P$ or the Hasse diagram of $P$. Then, whp
\[
\frac{c_d(\log n)^{d-1}}{(\log\log n)^{2d-2}}\le\chi(G)\le \frac{C_d(\log n)^{d-1}}{\log\log n},\qquad 
\frac{c_dn\log \log n}{(\log n)^{d-1}}\le \alpha(G)\le\frac{C_dn(\log \log n)^{2d-2}}{(\log n)^{d-1}},
\]
for some $c_d,C_d>0$ only depending on $d$.
\end{theorem}
We believe that, in \cref{thm:d>2}, the upper bound on the chromatic number and the lower bound on the independence number are both sharp up to constant factors (i.e., we conjecture that the correct exponent of $\log \log n$ is 1, for any fixed dimension $d$). Given that we were unable to prove this, we made no particular effort to optimise the exponent ``$2d-2$'' (however, in \cref{remark: improvable}, we sketch how it can be improved to $d-1$ ).

\subsection{Further directions}
As stated above, an obvious question left open by our work is to close the gaps between the upper and the lower bounds in \cref{thm:d>2}.

Actually, one might wonder whether (for every constant $d\ge2$) $(\log n)^{d-1+o(1)}$ is the maximum chromatic number of \emph{any} $n$-vertex box-Delaunay graph or Hasse diagram of dimension $d$, and $n/(\log n)^{d-1+o(1)}$ is the minimum independence number. 
However, this currently seems far out of reach: some of the most important open questions in this direction are whether the chromatic number of a 2-dimensional Hasse diagram or box-Delaunay graph is always $n^{o(1)}$, and whether the independence number is always $n^{1-o(1)}$. 
The current best general upper bounds on the chromatic number (implying lower bounds on the independence number) are due to Chan~\cite{Cha12}, based on ideas of Ajwani, Elbassioni, Govindarajan, and Ray~\cite{AEGR12}.

Another obvious direction is to try to better understand the dependence on $d$ (in \cref{thm:d>2}, we treat $d$ as a fixed constant). In particular, we wonder if it might be possible to improve bounds on the maximum chromatic number of an $n$-vertex Hasse diagram by considering a random set of $n$ points in $[0,1]^{d(n)}$, for a judicious choice of the function $d(n)$.

Finally, we remark that much of the work in computational geometry on box-Delaunay graphs is motivated by a problem called \emph{conflict-free colouring} (introduced by Even, Lotker, Ron and Smorodinsky~\cite{ELRS03}). Specifically, a conflict-free colouring of a point set $P\subseteq \mb R^d$ (with respect to axis-aligned boxes) assigns colours to the points in $P$ in such a way that, for every box $B\subseteq \mb R^d$ containing a point of $P$, there is a colour which appears exactly once in $B\cap P$. It is easy to see that conflict-free colourings of $P$ are proper colourings of the box-Delaunay graph of $P$, so our bounds on the chromatic number immediately imply bounds on the minimum number of colours in a conflict-free colouring. It would be interesting to investigate to what extent these implied bounds are sharp.

\subsection*{Notation}
We use standard asymptotic notation throughout, as follows. For functions $f=f(n)$ and $g=g(n)$ (with $g$ nonnegative), we write $f=O(g)$ to mean that there is a constant $C$ such that $|f(n)|\le Cg(n)$, and we write $f=\Omega(g)$ to mean that there is a constant $c>0$ such that $f(n)\ge cg(n)$ for sufficiently large $n$.
We write $f=\Theta(g)$ to mean that $f=O(g)$ and $g=\Omega(f)$, and
we also write $f=o(g)$ to mean that $f(n)/g(n)\to0$ as $n\to\infty$. Subscripts on asymptotic notation indicate quantities that should be treated as constants (though we will usually not write this explicitly for the dimension $d$, which we will always view as a constant).

For a real number $x$, the floor and ceiling functions are denoted $\lfloor x\rfloor=\max\{i\in \mb Z:i\le x\}$ and $\lceil x\rceil =\min\{i\in\mb Z:i\ge x\}$. We will however sometimes omit floor and ceiling symbols and assume large numbers are integers, wherever divisibility considerations are not important. 
In addition, all logarithms without an explicit base should be interpreted as having base $\e$. 
Finally, for any point $x \in \mb{R}^d$, we will use $x_i$ to denote its $i$-th coordinate.

\section{Proof outline}\label{sec:outline}

In this section, we outline some of the key aspects of the
proofs of \cref{thm:d=2,thm:d>2}, making comparisons to previous work (and explaining the limitations of some natural alternative ideas). This section also serves as a roadmap to lemmas proved in the rest of the paper.

Fix $d\ge2$ and let $P\subseteq[0,1]^{d}$ be a uniformly random set of $n$ points in $[0,1]^{d}$. 
Write $D(P)$ for the box-Delaunay graph of $P$, and write $H(P)$ for the Hasse diagram of $P$ (with respect to the usual coordinatewise partial order). 
Since $H(P)$ is a subgraph of $D(P)$, it holds that 
\[
    \alpha(H(P)) \ge \alpha(D(P)) \ge \frac{|P|}{\chi(D(P))}, \quad   
    \chi(D(P)) \ge \chi(H(P)) \ge \frac{|P|}{\alpha(H(P))}.
\]
So, to prove \cref{thm:d=2,thm:d>2}, it suffices to show that whp
\begin{align}
\chi(D(P)) & = O\left(\frac{(\log n)^{d-1}}{\log\log n}\right),\label{eq:chi}\\
\alpha(H(P)) & = O\left(\frac{n\log\log n}{\log n}\right)\text{ for }d=2,\label{eq:alpha-2}\\
\alpha(H(P)) & = O\left(\frac{n(\log\log n)^{2d-2}}{(\log n)^{d-1}}\right)\text{ for }d>2.\label{eq:alpha}
\end{align}
(We remind the reader that, throughout this paper, $d$ is always treated
as a constant for the purpose of asymptotic notation).

The upper bound on $\chi(D(P))$ in \cref{eq:chi} is proved in a completely
different way to the upper bounds on $\alpha(H(P))$ in \cref{eq:alpha-2}
and \cref{eq:alpha}; we outline both separately.

\subsection{Upper bounds on the chromatic number}

First, we outline the proof of \cref{eq:chi}.

\subsubsection{The maximum degree}

The first ingredient is the following bound on the maximum degree
of a random box-Delaunay graph.

\begin{lemma}\label{lemma: max degree}
    Fix $d \ge 2$.
    Let $P \subseteq [0,1]^d$ be a uniformly random set of $n$ points, and let $G$ be the box-Delaunay graph of $P$.
    Then, whp, the maximum degree of $G$ is $O((\log n)^{d-1})$.
\end{lemma}

It is a near-trivial fact that any graph with maximum degree $\Delta$
has chromatic number at most $\Delta+1$, so \cref{lemma: max degree} immediately implies
that whp $\chi(D(P)) = O((\log n)^{d-1})$, which differs only by
a $\log\log n$ factor from the desired bound in \cref{eq:chi}.

\begin{remark}\label{rem:avg-vs-max} In the setting of \cref{lemma: max degree}, it is much easier to study the
\emph{average} degree of $G$ than its maximum degree. Indeed, a straightforward
computation (summing over all pairs of vertices, using linearity of
expectation) shows that the expected average degree is $\Theta((\log n)^{d-1})$,
and one can show that the average degree is concentrated around its
expectation using the second moment method (this is done in \cite{CPST09}
for $d=2$). The so-called \emph{Caro--Wei bound} (see for example \cite[pp.~100--101]{AS16}) says that every $n$-vertex graph $G$ with average degree $\bar{d}$
has $\alpha(G)\ge n/(\bar{d}+1)$, so this is already enough to obtain
a nearly-optimal lower bound on the independence number.\end{remark}

We prove \cref{lemma: max degree} via a union bound over all vertices. For each vertex $v$,
we consider an exploration process that scans through the unit box
$[0,1]^{d}$, identifying large regions in which $v$ cannot possibly
have any neighbours while revealing as little information about $P$
as possible. Then, using the remaining
randomness of $P$, standard
concentration inequalities imply that, in the part of $[0,1]^{d}$ which
has not been ruled out, there are likely only $O((\log n)^{d-1})$ points of $P$.

In order to actually implement this kind of ``iterative exposure''
idea, it is important to use the technique of \emph{Poissonisation}:
namely, instead of studying our uniformly random $n$-point set $P$
directly, we consider a Poisson point process $\Pp$ with intensity $n$ 
and then, after proving the desired bounds, we relate the distributions
of $P$ and $\Pp$. The advantage of a Poisson process is that disjoint
regions of space are stochastically independent (so we can reveal
what happens in one region while retaining all the randomness of
a second region).

We collect some preliminary probabilistic tools in \cref{sec:prelims}, and the details
of the proof of \cref{lemma: max degree} appear in \cref{subsec:maxdegree}.

\subsubsection{Local sparsity}

To obtain the $\log\log n$ factor in \cref{eq:chi}, we need to use
more information about $D(P)$ than its maximum degree; namely, we
need to take advantage of \emph{local sparsity} 
information. For example,
it is easy to see that every Hasse diagram is triangle-free (equivalently,
for every vertex $v$, there are no edges among the neighbours of
$v$). This allows one to apply a well-known theorem of Johansson
\cite{Joh96b}\footnote{This paper is not easily available on the internet, but (an alternative presentation of) the proof can be found in \cite{MR02}. Similarly, a presentation of the proof in \cite{Joh96} can be found in \cite{BGG}.}, that every triangle-free graph with maximum degree $\Delta$ has
chromatic number at most $O(\Delta/\log\Delta)$. Indeed, Johansson's
theorem (together with \cref{lemma: max degree}) immediately implies \cref{eq:chi} with $H(P)$ in
place of $D(P)$.

Box-Delaunay graphs are not triangle-free in general, but it is easy
to see (e.g.\ using Ramsey's theorem or the Erd\H{o}s-Szekeres theorem) that, for every $d\ge 2$,
there is an integer $T$ such that $d$-dimensional box-Delaunay
graphs are $K_{T}$-free (i.e., contain no complete subgraph on $T$ vertices). 
Johansson also proved~\cite{Joh96} that $K_{T}$-free graphs with maximum
degree $\Delta$ have chromatic number $O_{T}(\Delta\log\log\Delta/\log\Delta)$; 
this result can be used to prove \cref{eq:chi} up to a factor of $\log\log\log n$.

\begin{remark}
    Parallelling \cref{rem:avg-vs-max}, we observe that, before Johansson, Shearer~\cite{Shearer95} famously proved that, for an $n$-vertex graph $G$ with average degree $\bar{d}$, if $G$ is triangle-free, then $\alpha(G) = \Omega(n\log\bar{d}/\bar{d})$ and, if $G$ is $K_{T}$-free, then $\alpha(G) = \Omega(n\log\bar{d}/(\bar{d}\log\log\bar{d}))$.
    The latter of these facts was used in \cite{CPST09} to prove a near-optimal lower bound on the independence number in the 2-dimensional case.
\end{remark}

It is conjectured that Johansson's bound for $K_{T}$-free graphs
can be improved by a factor of $\log\log \Delta$ but, as this is still
an open problem, we need to find another route to prove the precise bound in \cref{eq:chi}. Namely,
we use the following theorem of Alon, Krivelevich and Sudakov~\cite{AKS99} (which
they deduced from Johansson's theorem on triangle-free graphs).
\begin{theorem} \label{theorem: AKS}
    Let $G$ be a graph with maximum degree $\Delta$ where every vertex is incident to at most $\Delta^2/f$ triangles.
    Then, $\chi(G) = O(\Delta / \log f)$.
\end{theorem}
Note that, in a graph with maximum degree $\Delta$, every vertex is trivially incident to at most $\binom \Delta 2\le \Delta^2$ triangles. Informally speaking, \cref{theorem: AKS} says that whenever one has a significant improvement over this ``trivial bound'' (i.e., the neighbourhoods are ``sparse''), one gets a logarithmic improvement on the worst-case chromatic number.

So, in order to deduce \cref{eq:chi} from \cref{theorem: AKS} (and \cref{lemma: max degree}), it suffices to show the following.

\begin{lemma}\label{lemma: few triangles}
    Fix $d \ge 2$.
    Let $P\subseteq [0,1]^d$ be a uniformly random set of $n$ points, and let $G$ be the box-Delaunay graph of $P$.
    Then, whp, the number of triangles incident to every vertex is $O((\log n)^{2d-3+o(1)})$.
\end{lemma}

In our proof of \cref{lemma: few triangles}, we can afford to be rather crude (we believe the bound in \cref{lemma: few triangles} is far from best possible). 
We classify the edges (or pairs of vertices) in $D(P)$ to be ``far'' or ``close'' depending on the volume of the minimum axis-aligned box enclosing the two vertices of the edge. 
We show that, for every vertex $v$, there are typically not very many far edges involving $v$ (because, if two vertices are far, they are easily obstructed from being adjacent in $D(P)$ due to the presence of a vertex of $P$ ``between'' them).  
On the other hand, we show that there are typically not many triangles $vxy$ such that $vx,vy,xy$ are all close, for volume reasons. 
Indeed, for any two close vertices $v,x$, there is only a small region of space in which a third vertex $y$ could lie so that $(x,y)$ is a close pair, and $y$ is ``even closer'' to $v$ than $x$ is.
To actually prove this, we need to distinguish various cases for the relative positions of $v,x,y$.

The details of the proof of \cref{lemma: few triangles} appear in \cref{subsec:triangles}.
 
\subsection{Upper bounds on the independence number}
In this section, we describe how to prove \cref{eq:alpha-2} and \cref{eq:alpha}. 
As usual in the study of independence and chromatic numbers of random graphs, the overall strategy is to take a union bound over all sets of a suitable size $k$. 
For example, suppose we order the points in our random point set $P$ by their first coordinate and associate the points in $P$ with the integers $1,\dots,n$; we can then view the Hasse diagram $H(P)$ as a random graph on the vertex set $\{1,\dots,n\}$. 
Then, for each fixed set $I\subseteq\{1,\dots,n\}$ of size $k$, 
suppose we can show that there is an edge of $H(P)$ between vertices
of $I$ with probability at least $1-o(1/\binom{n}{k})$. This will
allow us to union bound over all $\binom{n}{k}$ choices of $I$ to show that whp $\alpha(H(P))<k$.

For a particular pair $i,j\in I$, it is not hard to compute the probability of the event $\mathcal{E}_{ij}$ that $ij$ forms an edge in $H(P)$ (for example, when $d=2$, this probability is exactly $1/(|i-j|+1)$). 
However, the edges of $H(P)$ are quite strongly dependent, and the main challenge is to deal with this dependence. It is informative to start by outlining the approach in \cite{CPST09}, as our proofs of \cref{eq:alpha-2} and \cref{eq:alpha} use it as a starting point.

\subsubsection{The Chen--Pach--Szegedy--Tardos approach}\label{subsubsec:CPST}
In \cite{CPST09}, the authors show that, when $d=2$, there is $k=O(n(\log\log n)^{2}/\log n)$ such that, for any fixed set $I\subseteq\{1,\dots,n\}$ of size $k$, we have $\Pr\big[\bigcup_{i,j\in I}\mathcal E_{ij}\big]\ge 1-o(1/\binom{n}{k})$.

First, note that, while the events $\mathcal{E}_{ij}$ are not mutually independent, it is possible to identify a collection of $\lfloor k/2\rfloor$ pairs $\{i,j\}$ such that the corresponding events $\mathcal{E}_{ij}$ are mutually independent; for example, write the elements of $I$ as $i_{1}<\dots<i_{k}$, and consider the ``consecutive'' pairs $\{i_{1},i_{2}\}$, $\{i_{3},i_{4}\}$, \dots, $\{i_{2\lfloor k/2\rfloor-1},i_{2\lfloor k/2\rfloor}\}$.
It is not hard to see that the \emph{expected} number of these pairs which form an edge in $H(P)$ is at least $\Omega(k\cdot(k/n))$ 
(because, no matter the choice of $I$, the average distance $|i_{q+1}-i_q|$ between consecutive elements $i_q,i_{q+1}\in I$ is about $n/k$), 
and one can show with a Chernoff bound that it is very unlikely that none of these pairs form an edge in $H(P)$: namely, this happens with probability $\exp(-\Omega(k^{2}/n))$. 
Unfortunately, this is too weak for a union bound over $\binom{n}{k}$ choices of $I$, no matter the value of $k$.

To go beyond this na\"ive bound, the authors of \cite{CPST09} observed
that it is actually not necessary to reveal the exact locations of
the points in $P$ to have a good chance of knowing whether a consecutive pair $\{i_{q},i_{q+1}\}$ forms an edge of $H(P)$. 
Indeed, writing $L=n/k$, one can
simply reveal the first digit of the $L$-ary expansions of the second coordinate of each point 
(roughly speaking, for $L$ random numbers $x(1),\dots,x(L)\in[0,1]$, knowing the first digits of their $L$-ary expansions is enough to tell us much of the information about the relative ordering between $x(1),\dots,x(L)$
).
So, the authors considered an ``iterative exposure'' scheme in which the $L$-ary digits of the (second coordinates of the) random points in $P$ are revealed one by one, 
and each time a new digit is revealed, we have a new opportunity to discover an edge of $H(P)$ between vertices of $I$. 
Specifically, before revealing the $r$-th digit, the vertices can be partitioned into $L^{r-1}$ ``levels'' according to their first $r-1$ digits. 
Then, when considering the effect of revealing the $r$-th digits, we can consider pairs of vertices $i,j\in I$ which are consecutive \emph{within their level} (i.e., $i,j$ are on the same level, and there are no other points of $I$ between $i,j$ in this level). 
As long as $L^{r-1}=o(k)$ (so, on average, the levels have many vertices of $I$ in them), we can argue similarly to the above paragraph to see that, at each digit revelation step, we fail to find an edge with probability $\exp(-\Omega(k^{2}/n))$.

Note that $L^{r-1}=o(k)$ as long as (say) $r\le(1/2)\log k/\log L$, 
so this argument can be used to show that there is an edge of $H(P)$ between two vertices of $I$, 
except with probability $\exp(-\Omega(k^{2}/n)\cdot(1/2)\log k/\log L)$.
A simple calculation shows that this probability is $o(1/\binom{n}{k})$ as long
as $k$ is a sufficiently large multiple of $n(\log\log n)^{2}/\log n$.

\subsubsection{Sharpening the analysis in the $2$-dimensional case}\label{subsubsec:sharper}

The analysis sketched above is inefficient in several respects. Most
obviously, at each digit revelation step, it considers only pairs of \emph{consecutive} vertices in $I$, within each level (we just call these ``consecutive pairs'').
It turns out that this inefficiency is not critical: intuitively, it is not too hard to imagine that if two vertices end up forming an edge, it is quite likely that they were consecutive at some point during the digit-revelation process.

More importantly, and somewhat less obviously, the above analysis makes a \emph{worst-case} assumption about the distances between consecutive pairs at each step (here the ``distance'' between two points on the same level is the number of points between them {\em on that level}, plus one). 
If all the $k$ vertices in $I$ were roughly equally distributed between the levels, and roughly equally spaced within each level, then the distance between the vertices of each consecutive pair would be about $n/k$. 
However, we can benefit from {\em fluctuations in these distances}. 
For example, it is reasonable to imagine that, for $i\le \log(n/k)/10$, say, about a $2^{-i}$-fraction of consecutive pairs are at distance about $2^{-i}(n/k)$. 
If this were the case at some point in the digit revelation process, then the expected number of consecutive pairs that would become an edge after the next digit revelation would be about $\sum_{i\le \log(n/k)/10}(2^{-i}k)/(2^{-i}n/k)$, which is of order $\log(n/k)k^{2}/n$. 
That is to say, we would gain an additional factor of $\log(n/k)$ compared to the pessimistic calculation in \cref{subsubsec:CPST}. 
If we were to blindly repeat the rest of the calculations with this additional $\log(n/k)$ factor, we would be able to obtain an optimal bound; 
that is to say, these (currently unjustified) calculations would show that whp $\alpha(H(P))=O(n\log\log n/\log n)$.

So, we need to study the distribution of the distances of consecutive pairs. 
We would like to consider some kind of ``typical distances event'', which occurs with probability $1-o(1/\binom{n}{k})$. 
Informally, this event should say that, ``on average'', the digit revelation process encounters distance fluctuations which yield the desired factor of $\log(n/k)$. 
We would then be able to study the digit revelation process as above, \emph{assuming} that our typical distances event holds, and our assumption will only have been incorrect with negligible probability $o(1/\binom{n}{k})$.

Now, we can study this typical distances event in a very similar way that we originally studied the event that $H(P)$ has an edge in $I$: namely, by revealing the second coordinates of our points in a digit-by-digit fashion, and considering consecutive pairs at each step.  
Crucially, here we consider \emph{binary} expansion, instead of base-$L$ expansion.  
Specifically, having revealed some binary digits (thereby partitioning the points into \emph{binary} levels), we can consider pairs of points in $I$ which are consecutive within their binary level. 
With probability $1/2$, such a pair will still be in the same binary level after revealing a further binary digit, and if this occurs we expect the ``distance'' between the two points in the pair to decrease by a factor of about $1/2$. 

Note that we can study the evolution of the distances of consecutive pairs (on binary levels) on a much more granular basis than we studied the emergence of edges (on $L$-ary levels). 
Indeed, if we assume that $L$ is a power of 2, then revealing an $L$-ary digit is tantamount to revealing a ``block'' of $\log_2 L$ binary digits. 
By revealing these $\log_2 L$ binary digits one by one, we can track the emergence and survival of consecutive pairs: the consecutive pairs which emerge early in the block are less likely to survive to the end of the block, but if they do survive, 
the distance between them  at the end tends to be proportionally smaller.
This allows us to study the distribution of distances of consecutive pairs at each $L$-ary digit revelation step, as required for the typical distances event.

At a very high level, this describes the key ideas in the proof of \cref{eq:alpha-2}. 
The details appear in \cref{sec:d=2}. 
As a few further technical remarks: first, since we need our typical distances event to hold with overwhelming probability, it is crucial that it concerns the \emph{average} behaviour of the
digit revelation process. 
In general, it is not sufficiently unlikely that there are \emph{some} steps where all points in $I$ are evenly spaced among $\{1,\dots,n\}$, and the distances between consecutive pairs are therefore all roughly equal.
Second, it is important that our typical distances event only concerns pairs of vertices which are reasonably far from each other (we can study pairs of vertices at distance at least $(n/k)^{1/10}$, say, but as we decrease the distances under consideration, pairs at such distances become too rare for concentration inequalities to provide the desired bounds).

\subsubsection{Higher dimensions}\label{subsubsec: d dimensions}
For \cref{eq:alpha}, we need to generalise the considerations in \cref{subsubsec:CPST} to $d>2$. There is a simple way to do this that does not really require any new ideas: we can start by fixing the first coordinates of all the points in $P$, and then generalise the digit revelation process to gradually reveal all other coordinates \emph{simultaneously}, at each step
revealing a new digit in all of the last $d-1$ coordinates. Unfortunately, while this approach provides stronger bounds for $d>2$ than for $d=2$, the improvement is limited to the exponent of $\log\log n$ (i.e., for any constant $d$, we obtain a
bound of the form $\alpha(H(P))\le n/(\log n)^{1+o(1)}$, whereas \cref{eq:alpha}
demands a bound of the form $\alpha(H(P))\le n/(\log n)^{d-1+o(1)}$)\footnote{The problem is that (the natural generalisation of) ``consecutive pair'' does not capture enough pairs of vertices. This is related to the unusual geometry of Hasse diagrams and box-Delaunay graphs: even if two vertices are very far in one of their coordinates, if they are very close in another coordinate then they are likely to form an edge.}.

To prove \cref{eq:alpha}, our approach is to first expose the first $d-1$ coordinates, and then consider the digit-by-digit revelation process only on the single last coordinate, adapting the
notion of ``consecutive pair'' to a notion of ``suitable pair''  that incorporates information from all the first $d-1$ coordinates (instead of just the first coordinate).
Borrowing the philosophy in \cref{subsubsec:sharper},  our definition of ``suitable pair'' does not need to guarantee the existence of many such pairs with probability 1; 
instead, we may work under a ``typical pairs event'' which holds with probability $1-o(1/\binom{n}{k})$ and ensures that enough suitable pairs exist throughout the digit
revelation process.

Recall that, in the analysis in \cref{subsubsec:CPST} for $d=2$, the key reason we restricted our attention to consecutive pairs was that they provided us with a very convenient source of \emph{mutual independence} in a highly dependent situation. 
Unfortunately, this very strong independence does not seem to be possible in our higher-dimensional setting, and our definition of ``suitable pair'' will be much less constrained.

To discuss the relevant complications, it is helpful to be a bit more concrete about our setting. 
Suppose that, for some appropriate $k$, we are interested in a particular subset\footnote{Formally speaking, we can view $P$ as a \emph{labelled} set of points with labels $\{1,\dots,n\}$, so it makes sense to ``specify a subset of $P$'' before revealing the randomness of $P$: this just means specifying a subset of the labels.} 
of $k$ of the points in $P$ (call these ``marked'' points); our main goal is to show that it is overwhelmingly likely that there is an edge between two marked points. 
Suppose that at some point we have fully revealed the first $d-1$ coordinates of every point of $P$, and we have revealed $r$ of the $L$-ary digits of the $d$-th coordinate of every point of $P$ (thereby partitioning the points into $L^r$ levels). 
For every pair of marked points $x,y$ on the same level, there is some subset $B_{x,y}\subseteq P$ of points on that level that are ``between'' $x$ and $y$ (according to the natural ordering on the first $d-1$ coordinates). 
This set $B_{x,y}$ 
is the set of all points that can influence whether $xy$ becomes an edge of $H(P)$.

In the $d=2$ case described in \cref{subsubsec:CPST}, at every digit-revelation step, consideration of consecutive pairs gave us a family of pairs $x,y$ for which the sets $B_{x,y}$ were completely disjoint. 
Unfortunately, in higher dimensions, we do not know whether it is possible to find such a family of pairs (we no longer have a good notion of ``consecutive'').

To sidestep this difficulty, our key observation is that (surprisingly) we can get quite far by simply ignoring the randomness of all the $n-k$ unmarked points of $P$, relying entirely on the randomness of our $k$ marked points. 
That is to say, when trying to bound the probability that our $k$ marked points form an independent set of $H(P)$, we can imagine that an adversary chooses the positions of the $n-k$ unmarked points, and only work with the randomness of the $k$ marked points. 
This means that, for mutual independence, we do not need our sets $B_{x,y}$ to be disjoint; we just need the sets $B_{x,y}$ not to contain any marked points except $x,y$, and we need the pairs $\{x,y\}$ themselves to be disjoint from each other.

It may be helpful to reconsider the $d=2$ analysis in \cref{subsubsec:CPST}, to see that the entire proof can be adapted to only use the randomness of the marked points. 
Indeed, our starting point was the observation that for a pair of marked points $i,j\in\{1,\dots,n\}$, the probability that they form an edge is about $1/|i-j|$. 
This is true even if we allow an adversary to first choose the positions of the $|i-j|-1$ points between $i$ and $j$: the second coordinates of these points divide $[0,1]$ into $|i-j|$ regions, and we simply need $i$ and $j$ to fall in the same region, which happens with probability at least $1/|i-j|$ by a simple convexity argument.

So, in our proof of \cref{eq:alpha}, we completely abandon the notion
of ``consecutive'', and simply work with arbitrary disjoint pairs of marked
vertices. 
At each step, our family of ``suitable pairs'' will be a collection of about $k$ disjoint pairs $\{x,y\}$ such that $|B_{x,y}|\le (n/k)/(\log n)^{d-2+o(1)}$.
The (overwhelmingly likely) existence of such pairs is established by induction on the dimension. 
Indeed, recall that an edge in $H(P)$ is a pair of $d$-dimensional points with no points ``between'' them, and we have explained how to study such pairs using digit revelation of the $d$-th coordinate via ``suitable pairs'' defined on the first $d-1$ coordinates. In much the same way, a suitable pair is a pair of $(d-1)$-dimensional points with \emph{few} points ``between'' them, and we can study such pairs using digit revelation of the $(d-1)$-th coordinate via a $(d-2)$-dimensional notion of ``suitable pairs''. 
Each time we reduce the dimension, the threshold for being considered a ``suitable pair'' increases by a factor of $(\log n)^{1-o(1)}$, until we reach a 1-dimensional problem where the desired bound is trivial.

At a very high level, this describes the ideas in the proof of \cref{eq:alpha}: it can be viewed as an inductive version of the Chen--Pach--Szegedy--Tardos scheme, using only the randomness of the $k$ marked points. The details appear in \cref{sec:d>2}. 
We remark that the improvement for the $d=2$ case described in \cref{subsubsec:sharper} seems to fundamentally depend on the randomness of the unmarked points, so it is not clear to us how to combine the ideas in this subsection with the ideas in \cref{subsubsec:sharper}. 
We believe that closing the gap in \cref{thm:d>2} would require combining ideas from both approaches.

\section{Probabilistic preliminaries}\label{sec:prelims}

In this short section, we collect a couple of basic probabilistic preliminaries.
First, in the proofs of \cref{lemma: max degree,lemma: few triangles}, we will make use of a technique known as \emph{Poissonisation}, which has many applications in geometric probability (see~\cite{Pen03} for a detailed account of the subject).
In essence, it replaces a binomial process on $n$ points in the cube $[0,1]^d$ with a homogeneous Poisson point process $\Pp\subseteq [0,1]^d$ with intensity $n$. For the reader who is not familiar with this notion, it is not too important to rigorously define Poisson point processes; we only need the following two key properties.
\begin{enumerate}
    \item[(i)] For every measurable set $\Lambda\subseteq [0,1]^d$, $|\Lambda\cap \Pp|$ is a Poisson random variable with parameter $n\cdot\vol(\Lambda)$, where $\vol(\Lambda)$ is the volume of $\Lambda$.
    \item[(ii)] For every pair of disjoint measurable sets $\Lambda_1, \Lambda_2\subseteq [0,1]^n$, the random variables $|\Lambda_1\cap \Pp|$ and $|\Lambda_2\cap \Pp|$ are independent.
\end{enumerate}

One key observation to apply the technique of Poissonisation is that, for a Poisson random variable $X$ with parameter~$n$,
\begin{equation}\label{eq:Pois}
\mathbb P(X = n) = \frac{\e^{-n} n^n}{n!} = \Theta\bigg(\frac{1}{\sqrt{n}}\bigg).
\end{equation}
In particular, every property that holds with probability $1-o(n^{-1/2})$ in the Poissonised model described above also holds whp in the binomial model on $n$ vertices.

Finally, we need some standard large deviation bounds. The following simple Chernoff bound for Binomial and Poisson random variables appears as, for example, \cite[Theorems~A.1.12 and~A.1.15]{AS16}.

\begin{theorem}\label{thm:chernoff}
For a Binomial or Poisson random variable $X$ with mean $\mu$ and any $\varepsilon>0$, we have
\[\Pr[X\le (1-\varepsilon)\mu]\le \exp(-\varepsilon^2\mu/2),\quad\qquad \Pr[X\ge (1+\varepsilon)\mu]\le (\e^\varepsilon (1+\varepsilon)^{-(1+\varepsilon)})^\mu.\]
In particular, the latter inequality implies
\begin{equation}\label{eq:standard}
\Pr[X > 2\mu] \le 2^{-\mu/2}.
\end{equation}
\end{theorem}
The following version of Bernstein's inequality is a consequence of, for example,  \cite[Theorem~2.8.1]{Vershynin18}.
\begin{theorem}\label{thm:bernstein}
    Let $X_1,\dots,X_n$ be i.i.d.\ exponential random variables with rate $r$ (and therefore, mean $1/r$). Then, for any $t\ge 0$,
    \[\Pr\Big[|X_1+\dots+X_n-n/r|\ge t\Big]\le 2\exp\left(-\Omega\left(\frac{t^2}{n/r^2+t/r}\right)\right).\]
\end{theorem}

\section{The local statistics}\label{sec: local statistics}

In this section, we prove \cref{lemma: max degree,lemma: few triangles}, concerning local statistics of random box-Delaunay graphs. As discussed in \cref{sec:outline}, these lemmas imply \cref{eq:chi}, which provides the upper bounds on chromatic number and lower bounds on independence number in \cref{thm:d=2,thm:d>2}.

We start with some notation.
\begin{definition}
We denote by $o$ the origin of $\mb{R}^d$. 
For any measurable subset $A$ of $\mb{R}^d$, we write $\vol(A)$ for its volume.
Given two points $x,y \in \mb{R}^d$, define $\mr R[x,y]$ to be the {\em open rectangle}
with endpoints $x,y$, that is, the set of points $z$ such that $\min(x_i,y_i) < z_i < \max(x_i,y_i)$ for all $i=1,2,\dots,d$. 
\end{definition}
We often assume that no two points share any of the coordinates and no point has any coordinates equal to 0 or 1 (this happens with probability 1). 
Also, we remind the reader that, throughout this paper, $d$ is always treated
as a constant for the purpose of asymptotic notation.

\subsection{The maximum degree}\label{subsec:maxdegree}
First, we prove \cref{lemma: max degree}. 
In fact, we prove a stronger Poissonised version. Recall that, if $x,y$ form an edge in $G$, then $\mr R[x,y]\cap P = \emptyset$.

\begin{lemma} \label{lemma: degree in poisson}
    For every $d\ge 2$, there exists a constant $C=C(d)>0$ such that the following holds.
    Let $\Pp$ be a Poisson point process with intensity $n$ on the space $[0,1]^d$, and let $X$ be the set of points $x \in \Pp$ such that $\mr R[o,x]\cap \Pp=\emptyset$.
    Then, $\Pr[|X| > C(\log n)^{d-1}] = o(n^{-2})$.
\end{lemma}

Before proving \cref{lemma: degree in poisson}, we show how to deduce \cref{lemma: max degree} from it.

\begin{proof}[Proof of \cref{lemma: max degree}]
    Let $C=C(d)>0$ be the constant in \cref{lemma: degree in poisson}.
    Denote by $p_1,\dots,p_n$ the $n$ uniformly random points in $P$.
    For every $i \in \{1,\ldots,n\}$, let $X_i$ be the set of points $p_j$, $j \neq i$, such that $\mr R[p_i,p_j] \cap P=\emptyset$. Note that $|X_i|$ is precisely the degree of $p_i$ in $G$, so it suffices to show that $\Pr[|X_i|> 2^d\cdot C(\log n)^{d-1}] = o(n^{-1})$ for each $i$.
    To this end, fix any $i$ and fix an outcome of the $i$-th point $p_i \in [0,1]^d$. Let $\Pp$ and $\Pp'$ be Poisson point processes with intensity $n$ on the spaces $[0,1]^d$ and $p_i+[-1,1]^d$, respectively. (Here, we write $p_i+[-1,1]^d$ for the set of points $x\in \mb R^d$ such that $x-p_i\in [-1,1]^d$).
    Define $$
        X_i(\Pp) := \left\{ x \in \Pp: \mr R[p_i,x]\cap \Pp = \emptyset \right\}\quad  \text{ and } \quad 
        X_i(\Pp') := \left\{ x \in \Pp': \mr R[p_i,x]\cap \Pp' = \emptyset \right\}.
    $$
    Observe that the distribution of $\{p_j:j\ne i\}$ is identical to that of $\Pp$ conditioned on the event $|\Pp| = n-1$, which happens with probability $\Theta(n^{-1/2})$ (recalling \cref{eq:Pois}).
    Using the fact that $|X_i(\Pp')|$ stochastically dominates $|X_i(\Pp)|$, we obtain
    \begin{equation} \nonumber
      \begin{aligned}
        \Pr\left[|X_i| > 2^d\cdot C(\log n)^{d-1} \big| p_i \right]
        &=\, \Pr\Big[|X_i(\Pp)| > 2^d \cdot C(\log n)^{d-1} \big| |\Pp|=n-1, p_i \Big] \\
        &\le\, \frac{\Pr\Big[|X_i(\Pp)| > 2^d\cdot C(\log n)^{d-1} \big| p_i \Big]}{\Pr\big[|\Pp|=n-1 \big]} \\
        &=\, O(\sqrt{n}) \cdot \Pr\left[|X_1(\Pp')| > 2^d\cdot C(\log n)^{d-1}  \big\vert p_i\right].
      \end{aligned}
    \end{equation}
    By applying \cref{lemma: degree in poisson} in all the $2^d$ orthants of $[-1,1]^d$, we see that $\Pr\left[|X_1(\Pp')| > 2^d\cdot C(\log n)^{d-1} \big| p_i\right] = o(n^{-2})$.
    Hence, $\Pr\!\left[|X_i| > 2^d\cdot  C(\log n)^{d-1} \big| p_i \right] = o(n^{-3/2})$. Since this is true for any outcome of $p_i$, the desired result follows.
\end{proof}

Before getting into the details of the proof of \cref{lemma: degree in poisson}, we sketch how it works when $d=2$.

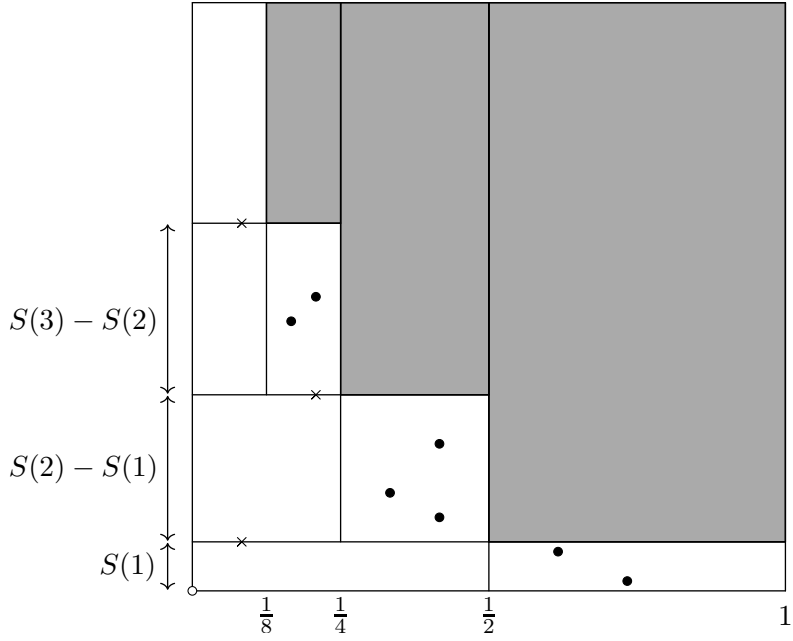
\begin{figure}
    \centering
    \begin{tikzpicture}[scale=0.65,line cap=round,line join=round,x=1cm,y=1cm]
        \clip(-19.5,-5.85) rectangle (12.247622119457375,7.1);
        \fill[line width=0.5pt,fill=black,fill opacity=0.1] (-7,7) -- (-7,-4) -- (-1,-4) -- (-1,7) -- cycle;
        \fill[line width=0.5pt,fill=black,fill opacity=0.1] (-10,7) -- (-10,-1) -- (-7,-1) -- (-7,7) -- cycle;
        % \fill[line width=0.5pt,fill=black,fill opacity=0.1] (-11.5,7) -- (-11.5,2.5) -- (-10,2.5) -- (-10,7) -- cycle;
        \draw [fill={rgb:black,4;white,8}] (-11.5,2.5) rectangle (-10,7);
        \draw [fill={rgb:black,4;white,8}] (-10,-1) rectangle (-7,7);
        \draw [fill={rgb:black,4;white,8}] (-7,-4) rectangle (-1,7);

        \draw [line width=0.5pt] (-13,7)-- (-13,-5);
        \draw [line width=0.5pt] (-13,-5)-- (-1,-5);
        \draw [line width=0.5pt] (-13,-4)-- (-1,-4);
        \draw [line width=0.5pt] (-1,-4)-- (-1,-5);
        \draw [line width=0.5pt] (-7,-4)-- (-7,7);
        \draw [line width=0.5pt] (-7,7)-- (-1,7);
        \draw [line width=0.5pt] (-1,7)-- (-1,-4);
        \draw [line width=0.5pt] (-7,-4)-- (-7,-5);
        \draw [line width=0.5pt] (-13,7)-- (-7,7);
        \draw [line width=0.5pt] (-7,7)-- (-7,-4);
        \draw [line width=0.5pt] (-7,-4)-- (-1,-4);
        \draw [line width=0.5pt] (-1,-4)-- (-1,7);
        \draw [line width=0.5pt] (-1,7)-- (-7,7);
        \draw [line width=0.5pt] (-13,-1)-- (-10,-1);
        \draw [line width=0.5pt] (-10,-1)-- (-7,-1);
        \draw [line width=0.5pt] (-10,-1)-- (-10,-4);
        \draw [line width=0.5pt] (-10,7)-- (-10,-1);
        \draw [line width=0.5pt] (-10,7)-- (-10,-1);
        \draw [line width=0.5pt] (-10,-1)-- (-7,-1);
        \draw [line width=0.5pt] (-7,-1)-- (-7,7);
        \draw [line width=0.5pt] (-7,7)-- (-10,7);
        \draw [line width=0.5pt] (-11.5,-1)-- (-11.5,2.5);
        \draw [line width=0.5pt] (-11.5,2.5)-- (-13,2.5);
        \draw [line width=0.5pt] (-11.5,2.5)-- (-10,2.5);
        \draw [line width=0.5pt] (-11.5,2.5)-- (-11.5,7);
        \draw [line width=0.5pt] (-11.5,7)-- (-11.5,2.5);
        \draw [line width=0.5pt] (-11.5,2.5)-- (-10,2.5);
        \draw [line width=0.5pt] (-10,2.5)-- (-10,7);
        \draw [line width=0.5pt] (-10,7)-- (-11.5,7);
        \draw [<->,line width=0.5pt] (-13.5,-4.98) -- (-13.5,-4.02);
        \draw [<->,line width=0.5pt] (-13.5,-3.98) -- (-13.5,-1.02);
        \draw [<->,line width=0.5pt] (-13.5,-0.98) -- (-13.5,2.48);
        \begin{scriptsize}
            \draw [color=black] (-12,2.5)-- ++(-2.5pt,-2.5pt) -- ++(5pt,5pt) ++(-5pt,0) -- ++(5pt,-5pt);
            \draw [color=black] (-10.5,-1)-- ++(-2.5pt,-2.5pt) -- ++(5pt,5pt) ++(-5pt,0) -- ++(5pt,-5pt);
            \draw [fill=white] (-13,-5) circle (2.5pt);
            \draw [fill=black] (-8,-2) circle (2.5pt);
            \draw [fill=black] (-9,-3) circle (2.5pt);
            \draw [fill=black] (-8,-3.5) circle (2.5pt);
            \draw [fill=black] (-11,0.5) circle (2.5pt);
            \draw [fill=black] (-10.5,1) circle (2.5pt);
            \draw [fill=black] (-5.6,-4.2) circle (2.5pt);
            \draw [fill=black] (-4.2,-4.8) circle (2.5pt);
            \draw [color=black] (-12,-4)-- ++(-2.5pt,-2.5pt) -- ++(5pt,5pt) ++(-5pt,0) -- ++(5pt,-5pt);
            \draw[color=black] (-14.3,-4.5) node {\large{$S(1)$}};
            \draw[color=black] (-15.2,-2.5) node {\large{$S(2)-S(1)$}};
            \draw[color=black] (-15.2,0.5) node {\large{$S(3)-S(2)$}};
            \draw[color=black] (-1,-5.5) node {\large{$1$}};
            \draw[color=black] (-7,-5.5) node {\large{$\tfrac{1}{2}$}};
            \draw[color=black] (-10,-5.5) node {\large{$\tfrac{1}{4}$}};
            \draw[color=black] (-11.5,-5.5) node {\large{$\tfrac{1}{8}$}};
        \end{scriptsize}
    \end{tikzpicture}
    \caption{An illustration of the 2-dimensional case. The points $x(1)$, $x(2)$, $x(3)$ are denoted by crosses,
    and other points that could possibly contribute to $X$ (i.e., points $x\in \Pp$ with $\mr R[o,x]\cap \Pp=\emptyset$) are denoted by black dots. In this example, only the bottom-most two dots actually contribute to $X$.
    The points $x(1),x(2),x(3)$ ensure that no point in the grey area contributes to $X$.}    \label{figure: max degree}
\end{figure}

\vspace{0.5em}

We consider an exploration process that gradually reveals some of
the points in our Poisson process $\Pp\in[0,1]^{2}$, by ``sweeping''
a gradually shrinking horizontal line segment from the bottom of
$[0,1]^{2}$ to the top. This process is illustrated in \cref{figure: max degree}. 

First, we start with the line segment of length $1/2$ between the points $o=(0,0)$ and $(1/2,0)$. 
Imagine that this segment continuously sweeps upwards until it encounters a point $x(1)\in \Pp$. 
At this moment, we shrink our sweeping segment by a factor of a half, so it is now a line segment between\footnote{Recall that $x(1)_2$ stands for the second coordinate of the point $x(1) \in \mb{R}^2$.} $(0,x(1)_{2})$ and $(1/4,x(1)_{2})$. 
This segment continues to sweep upwards until it encounters the next point $x(2)\in \Pp$ (so $x(2)_{1}\le1/4$). 
We again shrink our sweeping segment by a factor of a half (so it is now a line segment between $(0,x(2)_{2})$ and $(1/8,x(2)_{2})$), and so on. 
We continue this process until the sweeping segment reaches the top of the square $[0,1]$. 
Write $x(1),\dots,x(m)$ for the sequence of points encountered by the process.

The purpose of this exploration process is that it defines large regions of the
square $[0,1]^{2}$ which cannot possibly contribute to our random
variable $X$, while revealing very little information about $\Pp$
(here, it is crucial that we are working with a Poisson process, so the
behaviour within disjoint regions is completely independent). Indeed,
writing $S(0)=0$ and $S(i)=x(i)_{2}$ for $i\in \{1,\dots,m\}$, we make several observations.
\begin{itemize}
    \item Note that regions of the form $(2^{-i},1]\times (S(i),1]$ cannot possibly contribute to $X$ since, for any point $x\in (2^{-i},1]\times (S(i),1]$, the rectangle $\mathrm{R}[o,x]$ contains $x(i)\in P_{i}$. 
    The union of these regions is illustrated in grey in \cref{figure: max degree}.     \item We have revealed all the points of $\Pp$ inside regions of the
    form $[0,2^{-i}]\times[0,S(i)]$, and in the region $[0,2^{-m}]\times [S(m),1]$. Namely, these are the points $x(1),\dots,x(m)$. 
        The increments $S(i)-S(i-1)$ are exponentially distributed with mean $2^{i}/n$,        and a standard concentration inequality can be used to show that we are very likely to have $m = O(\log n)$.
    \item This just leaves the regions of the form $(2^{-i},2^{-(i-1)}]\times(S(i-1),S(i)]$.
    It is easy to describe the total area $A$ of these regions (namely,
    $A$ can be bounded by a certain weighted sum of independent exponential
    distributions, which turns out to have mean $O(\log n/n)$). Our exploration process has
    revealed nothing at all about the points of $\Pp$ inside these
    regions, so the number of such points conditionally has a Poisson distribution with
    mean $A$. With standard concentration inequalities, we can show that
    this number, and hence $|X|$, is very likely to be $O(\log n)$.
\end{itemize}

For general $d \ge 2$, we would like to generalise the above process by sweeping a gradually shrinking $(d-1)$-dimensional box perpendicular to the $d$-th coordinate.
For $d\ge 3$, there are multiple different dimensions which may be shrunk by a factor of a half, and we have to consider all such choices ``in parallel''. 
That is to say, instead of a sequence of points $x(i)\in [0,1]^2$ for an integer $i\ge 1$ (where we keep track of the second coordinate $x(i)_2=S(i)\in [0,1]$), we have a \emph{$(d-1)$-dimensional array} of points $x(i_1,\dots,i_{d-1})\in [0,1]^d$, for integers $i_1,\dots,i_{d-1}\ge 1$ (and we keep track of the $d$-th coordinate $x(i_1,\dots,i_{d-1})_{d}=S(i_1,\dots,i_{d-1})\in [0,1]$). 
The most convenient way to describe this formally is via a \emph{recursive} exploration process, where $x(i_1,\dots,i_{d-1})$ is defined in terms of the $x(j_1,\dots,j_{d-1})$ for which $(j_1,\dots,j_{d-1})$ ``precedes'' $(i_1,\dots,i_{d-1})$.

For convenience, write $(i_1,\dots,i_{d-1}) \preceq (j_1,\dots,j_{d-1})$ if $i_k\le j_k$ for all $k\in \{1,\dots,d-1\}$, and write $(i_1,\dots,i_{d-1}) \prec (j_1,\dots,j_{d-1})$ if moreover $(i_1,\dots,i_{d-1})\neq (j_1,\dots,j_{d-1})$. For each $i_1,\dots,i_d\ge 1$:
\begin{itemize}
    \item Let\footnote{In defining $S'(1,\dots,1)$, we use the convention that the maximum of the empty set is zero.} $$
            S'(i_1,\dots,i_{d-1}) := \max_{(j_1,\dots,j_{d-1})\prec (i_1,\dots,i_{d-1})} S(j_1,\dots, j_{d-1}).
        $$
        This can be interpreted as the ``furthest along the $d$-th coordinate we have reached so far''.
    \item Let $x(i_1,\dots,i_{d-1})$ be the point 
        \[x \in \Pp \cap \Big([0,2^{-i_1}]\times\dots\times [0,2^{-i_{d-1}}]\times (S'(i_1,\dots,i_{d-1}),1]\Big)
        \]
        with the minimum possible value of $x_d$, and let $S(i_1,\dots,i_{d-1})$ be this value of $x_d$. (If no such $x$ exists, let $S(i_1,\dots,i_{d-1})=1$ and leave $x(i_1,\dots,i_{d-1})$ undefined).
        \item Let $T(i_1,\dots,i_{d-1}):= S(i_1,\dots,i_{d-1})-S'(i_1,\dots,i_{d-1})$.
\end{itemize}

According to this procedure, given $S'(i_1,\dots,i_{d-1})$, there exists an exponential random variable $T$ with mean $2^{i_1+\dots+i_{d-1}}/n$ such that $T(i_1,\dots,i_{d-1})= \min(T, 1- S'(i_1,\dots,i_{d-1}))$.
Additionally, the points in $X$ which are not of the form $x(i_1,\dots,i_{d-1})$ are guaranteed to lie in a certain region of $[0,1]^d$ described as follows (this corresponds to the unshaded region in \cref{figure: max degree}).

\begin{claim} \label{claim: partition of the cube}
Write $R(i_1,\dots,i_{d-1})$ for the box $[0,2^{-i_1}]\times\dots\times [0,2^{-i_{d-1}}]$. Then, with probability $1$, the points in $X \setminus \left\{x(i_1,\dots,i_{d-1}): i_1,\dots,i_{d-1}\ge 1\right\}$ all lie in 
    \[               
        \bigcup_{j_1,\ldots,j_{d-1}\ge 1} 
            \big( R(j_1-1,\dots,j_{d-1}-1)\setminus R(j_1,\dots,j_{d-1}) \big)
            \times \big( S'(j_1,\dots,j_{d-1}),S(j_1,\ldots,j_{d-1}) \big].
    \]
\end{claim}
\begin{proof}
    Consider any $x\in X$ which is not one of the $x(i_1,\dots,i_{d-1})$. We show that $x$ lies in the above region.
    
    First, for all $k \in \{1,\dots,d\}$, let $i_k$ the unique integer such that $2^{-i_k} < x_k\le 2^{1-i_k}$ (recall that we may assume all coordinates of points in $\Pp$ are strictly between zero and one).
    
    If we had $x_d > S(i_1,\dots,i_{d-1})$, then $x(i_1,\dots,i_{d-1})$ would lie in $\mr R[o,x]$, which would contradict our assumption that $x \in X$, causing a contradiction. Hence, $x_d \le S(i_1,\dots,i_{d-1})$.
    
    Let $(j_1,\dots,j_{d-1}) \preceq (i_1,\dots,i_{d-1})$ be such that $S(j_1,\dots,j_{d-1})$ is minimised subject to the condition $x_d \le S(j_1,\dots,j_{d-1})$.
    This minimality implies $x_d > S'(j_1,\dots,j_{d-1})$, so 
    \begin{equation}\label{eq:xinR}
        \begin{split}
            x \in R(i_1-1,\dots,i_{d-1}-1)\times 
            &(S'(j_1,\dots,j_{d-1}), S(j_1,\dots,j_{d-1})]\\
            &\subseteq R(j_1-1,\dots,j_{d-1}-1)\times 
            (S'(j_1,\dots,j_{d-1}), S(j_1,\dots,j_{d-1})].
        \end{split}
    \end{equation}
    Finally, observe that $x \notin R(j_1,\dots,j_{d-1})\times (S'(j_1,\dots,j_{d-1}), S(j_1,\dots,j_{d-1})]$. 
    Otherwise, this would mean that $x=x(j_1,\dots,j_{d-1})$, contradicting the definition of $x$. 
        \end{proof}

The next claim tells us that, with very high probability, once one of the $i_k$ reaches $2\ceil{3\log_2 n}$, we have finished sweeping through the entire hypercube $[0,1]^d$.

\begin{claim}\label{claim: finite steps of max degree}
    Let $m=2\ceil{3\log_2 n}$. With probability $1-n^{-\omega(1)}$, for every $(d-1)$-tuple of integers $(i_1,\ldots,i_{d-1})\in \{1,\ldots,m\}^{d-1}$ such that $\max(i_1,\ldots,i_{d-1}) = m$, it holds that $S(i_1,\ldots,i_{d-1}) = 1$.
\end{claim}
\begin{proof}
    It is sufficient to show that $\Pr[S(m,1,\ldots,1) < 1] < n^{-\omega(1)}$: indeed, the minimum of 
    \begin{equation}\label{eq: 11m}
        S(m,1,\ldots,1),\; S(1,m,1,\ldots,1), \ldots,\; S(1,\ldots,1,m)  
    \end{equation}
    bounds from below each $S(i_1,\ldots,i_{d-1})$ with $\max(i_1,\ldots,i_{d-1}) = m$, and the $d-1$ random variables in~\eqref{eq: 11m} all have the same distribution.
    In order to bound $\Pr[S(m,1,\dots,1)<1]$ from above, we observe that there is a family of independent exponential random variables $(T_i')_{i=1}^m$ with rates $2^{2-d-i}n$, respectively, such that $S(i,1,\dots,1)=\min(1,T_1'+T_2'+\dots+T_i')$ for every $i \in \{1,\ldots,m\}$.
    Notice that $2^{1-d-m/2}n \le n^{-1}$ and, hence, the probability that $T_{m/2+1}',\dots,T_m'$ are all less than 1 is bounded from above by $(1-\e^{-1/n})^{m/2} = n^{-\omega(1)}$.
    Therefore, $\Pr[S(m,1\dots,1) < 1] = \Pr[T_1'+\dots+T_m' < 1] = n^{-\omega(1)}$, as desired.
\end{proof}
We are now sufficiently equipped to prove \cref{lemma: degree in poisson}.
\begin{proof}[Proof of \cref{lemma: degree in poisson}]
    Recall that we assume all points in $\Pp$ lie in $(0,1)^d$ and no two points share any coordinate. 
    Let $m=2\ceil{3\log_2 n}$. 
    There are at most $m^{d-1}$ points of the form $x(i_1,\dots,i_{d-1})$, where $i_1,\dots,i_{d-1}\le m$, so by \cref{claim: partition of the cube,claim: finite steps of max degree}, with probability $1-n^{-\omega(1)}$, we have
    \[|X|\le m^{d-1}+\sum_{1 \le i_1,\dots,i_{d-2} \le m} |X(i_1,\dots,i_{d-2})|,\]
    where 
    $X(i_1,\dots,i_{d-2}) := \bigcup_{j=1}^m (X\cap \Lambda_j)$ with $\Lambda_j = \Lambda_j(i_1,\ldots,i_{d-2})$ defined to be the set 
    \begin{equation*}
        \big(R(i_1-1,\dots,i_{d-2}-1,j-1)\setminus R(i_1,\dots,i_{d-2},j)\big)\times \big(S'(i_1,\ldots,i_{d-2},j), S(i_1,\ldots,i_{d-2},j)\big].
    \end{equation*}

We will study each $|X(i_1,\dots,i_{d-2})|$ separately. Fix any $(i_1,\dots,i_{d-2})\in \{1,\dots,m\}^{d-2}$, and recall the definition $T(i_1,\dots,i_{d-2},j)=S(i_1,\dots,i_{d-2},j)-S'(i_1,\dots,i_{d-2},j)$. There is a family of independent exponential random variables $(T_j')_{j=1}^m$ with rate $n$ which may be coupled with the process so that $T_j' \ge 2^{-(i_1+\dots+i_{d-2}+j)}T(i_1,\dots,i_{d-2},j)$ for each $j \in \{1,\dots,m\}$.
        By Bernstein's inequality (\cref{thm:bernstein}),
        there exists an absolute constant $c > 1$ such that $\Pr[T_1'+\dots+T_m' > cm/n] < n^{-3}$.
                        Moreover, conditionally on $S(i_1,\dots,i_{d-2},j)$ and $S'(i_1,\dots,i_{d-2},j)$ for $j\in \{1,\dots,m\}$,
    the random variables $\big(|\Pp \cap \Lambda_j| \big)_{j=1}^m$ are independent Poisson variables with rate  
                            $$
        (2^{d-1}-1)2^{-(i_1+\dots+i_{d-2}+j)} T(i_1,\dots,i_{d-2},j) n
        \le 2^{d-1} T_j' n.
    $$
            Hence, $|X(i_1,\dots,i_{d-2})|$ is stochastically dominated by a Poisson random variable with parameter $\mu:=\sum_{j=1}^m 2^{d-1} T_j' n$.
    By~\eqref{eq:standard}, conditionally on the event $\mu \le 2^{d-1}cm$ (which was shown above to happen with probability at least $1-n^{-3}$), the probability that $|X(i_1,\dots,i_{d-2})| > 2^d cm$ is at most $2^{-m/2}\le n^{-3}$.
    The lemma then follows by a union bound over all choices of $(i_1,\dots,i_{d-2}) \in \{1,\ldots,m\}^{d-2}$.
\end{proof}

\subsection{Number of triangles incident to each vertex.}\label{subsec:triangles}
Now, we prove \cref{lemma: few triangles}.
We first give a separate proof for $d=2$, where we actually do not need any randomness. Similar observations have been made by various other authors~\cite{H-PS03,PT03}. \begin{proof}[Proof of \cref{lemma: few triangles} for $d=2$]
    Fix any point set $P$ containing points with distinct first coordinates and distinct second coordinates.
    Fix any $p \in P$.
    It suffices to show the number of edges of $G$ in the neighbourhood $N(p)$ is $O(\deg(p))$.
    Then, with the help of \cref{lemma: max degree}, the conclusion is verified.
    
    First of all, fix a coordinate system centered at $p$ and look at other points from left to right.
    Then, the $y$-coordinates of the neighbours of $p$ in the first and in the third quadrant form two decreasing sequences, while the $y$-coordinates of its neighbours in the second and in the fourth quadrant form two increasing sequences; see \cref{figure: four quadrants}.
    Thus, every point in $N(p)$ has at most two common neighbours with $p$ in its quadrant, so there are at most $|N(p)|=\deg(p)$ such edges.
    
    Since no edge goes between opposite quadrants, we are left to bound from above the number of edges between any pair of neighbouring quadrants. 
    Due to symmetry, it suffices to consider the first and the fourth quadrant. 
    Every point in one of these quadrants has at most one neighbour with smaller $x$-coordinate in the other quadrant, so the number of edges between the two quadrants is $O(|N(p)|)=O(\deg(p))$.
    Applying the same argument to the other three pairs of neighbouring quadrants finishes the proof.
\end{proof}
\begin{figure}
    \centering
    \definecolor{cqcqcq}{rgb}{0.7529411764705882,0.7529411764705882,0.7529411764705882}
    \begin{tikzpicture}[scale=0.6,line cap=round,line join=round,x=1cm,y=1cm]
        \clip(-15,-4.6) rectangle (12.178965932646557,4.6);
        \draw [line width=0.8pt] (-3,0)-- (4,0);
        \draw [line width=0.8pt] (-3,0)-- (-3,5);
        \draw [line width=0.8pt] (-3,0)-- (-3,-5);
        \draw [line width=0.8pt] (-3,0)-- (-10,0);
        \draw [line width=0.4pt] (0,5)-- (0,-5);
        \draw [line width=0.4pt] (-10,2)-- (4,2);
        
        \draw [line width=0.4pt, dashed] (3,1)-- (3,0);
        \draw [line width=0.4pt, dashed] (2,-1)-- (2,0);
        \draw [line width=0.4pt, dashed] (1,-2)-- (1,0);
        
        \begin{scriptsize}
            \draw [fill=white] (-3,0) circle (2.5pt);
            \draw [fill=black] (-2,5) circle (2.5pt);
            \draw [fill=black] (-1.5,4) circle (2.5pt);
            \draw [fill=cqcqcq] (-1,3) circle (2.5pt);
            \draw [color=black] (0,2)-- ++(-2.5pt,-2.5pt) -- ++(5pt,5pt) ++(-5pt,0) -- ++(5pt,-5pt);
            \draw [fill=cqcqcq] (3,1) circle (2.5pt);
            \draw [fill=black] (-3.5,4.5) circle (2.5pt);
            \draw [fill=cqcqcq] (-4,3.5) circle (2.5pt);
            \draw [fill=cqcqcq] (-6,1.5) circle (2.5pt);
            \draw [fill=black] (-7.5,0.5) circle (2.5pt);
            \draw [fill=cqcqcq] (2,-1) circle (2.5pt);
            \draw [fill=cqcqcq] (1,-2) circle (2.5pt);
            \draw [fill=cqcqcq] (-0.5,-2.5) circle (2.5pt);
            \draw [fill=black] (-1.5,-4) circle (2.5pt);
            \draw [fill=black] (-5,-1.5) circle (2.5pt);
            \draw [fill=black] (-4,-3.5) circle (2.5pt);
            \draw [fill=black] (-8.5,-0.5) circle (2.5pt);
            
            \draw[color=black] (-3.3,0.3) node {\large{$p$}};
            \draw[color=black] (3-3.3,2.3) node {\large{$q$}};
                    \end{scriptsize}
    \end{tikzpicture}
    \caption{The solid points are neighbours of $p$ and the grey points are the common neighbours of $p$ and $q$.}
    \label{figure: four quadrants}
\end{figure}
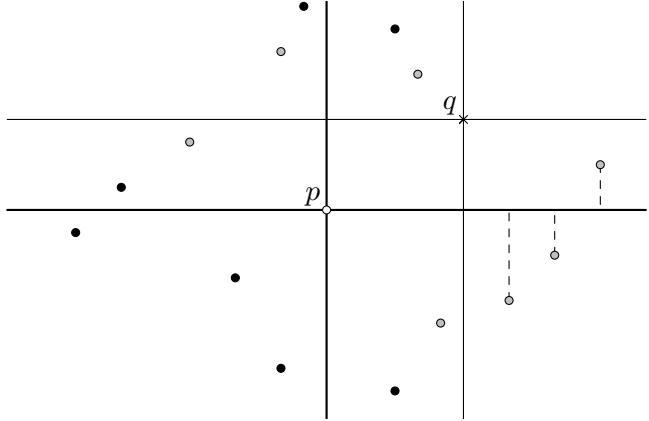

In the rest of this section, we prove \cref{lemma: few triangles} for $d \ge 3$.
Fix \begin{equation}\zeta := 2^{2d+1}d\frac{\log_2\log_2 n}{n}.\label{eq:zeta}\end{equation}
A pair of two points $x,y \in \mb{R}^d$ is said to be \emph{far} if $\vol(\mr R[x,y]) > \zeta$, and \emph{close} otherwise.
We categorize the edges of the box-Delaunay graph into {\em far edges} (if the two points form a far pair) and {\em close edges} (if the two points form a close pair).
Then, \cref{lemma: few triangles} is a consequence of \cref{lemma: max degree} and the following two statements.

\begin{lemma}\label{lemma: edges of far pairs}
    Fix $d \ge 2$.
    Let $P \subseteq [0,1]^d$ be a uniformly random set of $n$ points, and let $G$ be the box-Delaunay graph of $P$.
    Then, whp, every vertex in $G$ is incident to $O((\log n)^{1+o(1)})$ far edges.
\end{lemma}
\begin{lemma}\label{lemma: few triangles from close pairs}
    Fix $d \ge 3$.
    Let $P \subseteq [0,1]^d$ be a uniformly random set of $n$ points, and let $G$ be the box-Delaunay graph of $P$. 
        Then, whp, every vertex in $G$ is incident to $O((\log n)^{2d-3+o(1)})$ triangles formed by three close edges.
\end{lemma}

To deduce \cref{lemma: few triangles} for $d\ge 3$, note that \cref{lemma: max degree,lemma: edges of far pairs} imply that whp every vertex is incident to $O((\log n)^{d+o(1)})$ triangles with at least one far edge, and \cref{lemma: few triangles from close pairs} ensures that whp every vertex is incident to $O((\log n)^{2d-3+o(1)})$ triangles with three close edges.

We now prove \cref{lemma: edges of far pairs,lemma: few triangles from close pairs}.
The following lemma is a Poissonised version of \cref{lemma: edges of far pairs}. (One can deduce \cref{lemma: edges of far pairs} from this statement in the same way in which we deduced \cref{lemma: max degree} from \cref{lemma: degree in poisson}; we omit the details.)
\begin{lemma}\label{lemma: edges of far pairs in Poisson}
    Fix $d \ge 2$.
    Let $\Pp$ be a Poisson point process with intensity $n$ on the space $[0,1]^d$.
    Then, with probability $1-O(n^{-2})$, the number of points $x \in \Pp$ such that $(o,x)$ is a far pair and $\mr R[o,x]\cap \Pp=\emptyset$ is $O(\log n\log\log n)$.
\end{lemma}
\begin{proof}
    We may and will assume $n$ is sufficiently large, no point of $\Pp$ has any of its coordinates equal to 0 or 1, and no two points of $\Pp$ have the same value of any coordinate, whenever needed.

    First, we divide the points $x\in (0,1)^d$ into regions according to the likelihood that $\mr R[o,x]\cap \Pp=\emptyset$. For every $i_1,\dots,i_d \ge 1$, let $R(i_1,\dots,i_d)$ be the box $\prod_{k}[2^{-i_k}, 2^{-i_k+1})$.
    For every $I \ge d$, let $\mc{R}_I$ be the set of all boxes $R(i_1,\dots,i_d)$ such that $i_1+\dots+i_d = I$.

    Now, for any $x\in (0,1)^d$, note that $x$ lies in a unique box $R(i_1,\dots,i_d)$, namely $i_k=\ceil{\log_2 x_k^{-1}}$. 
    For convenience, we call $i_1+\dots+i_d=\sum_{k=1}^d \ceil{\log_2 x_k^{-1}}$ the {\em weight} of $x$. 
    Observe that if $(o,x)$ is a far pair, then $\prod_k 2^{-i_k+1} > \zeta$, i.e., the weight of $x$ is less than $\log_2\zeta^{-1} + d$.
    
    We say that $R(i_1,\dots,i_d)$ is \emph{empty} if $R(i_1,\dots,i_d) \cap \Pp=\emptyset$.
    Note that if $\mr R[o,x]\cap \Pp=\emptyset$, then, in particular, $R(i_1+1,\dots,i_d+1)$ is empty. So, to prove \cref{lemma: edges of far pairs in Poisson}, it will be helpful to show that there are typically not too many empty boxes.
    \begin{claim}\label{cl:3.9}
        For every $I < \log_2\zeta^{-1} + 2d$, with probability $1-O(n^{-3})$, the number of empty boxes in $\mc{R}_I$ is at most $2^{I+3}\log_2 n / n$.
    \end{claim}
        \begin{claimproof}[Proof of \cref{cl:3.9}]
        Write $s:=\ceil{2^{I+3}\log_2 n/n}$.
        Observe that there are $\binom{I-1}{d-1}\le I^d$ boxes $R(i_1,\dots,i_{d})$ in $\mc{R}_I$ and they are all disjoint.
        For any $s$ of these boxes, the total volume is $s/2^{I}$ and, thus,
        the probability that all of them are empty is $\exp(-ns2^{-I})$.
        Hence, the probability that at least $s$ of the boxes in $\mc R_I$ are empty is at most $p:=I^{ds}\cdot \exp(-ns2^{-I})$.
        Note that the factor ``$I^{ds}$'' does not significantly affect this expression: indeed, using that $I < \log_2\zeta^{-1}+2d \le \log_2 n$ (recalling the definition of $\zeta$ in \cref{eq:zeta}), we have 
        \[
            I^{ds} \le (\log_2 n)^{ds}
            = \exp(ds\log\log_2 n)
            \le \exp(ds\log_2\log_2 n)=\exp(ns\cdot \zeta/2^{2d+1})\le \exp(ns/2^{I+1}).
        \]
        This means $p \le \exp(-ns/2^{I+1})$.
        Since $s\ge 2^{I+3}\log_2 n/n$, we deduce that $p\le  \exp(-ns/2^{I+1}) \le \exp(-4\log_2 n) < n^{-3}$.
    \end{claimproof}

    Next, we observe that, with probability $1-O(n^{-2})$, there is no $x \in \Pp$ of weight at most $\log_2 n-2\log_2\log_2 n$ such that $\mr R[o,x]\cap \Pp=\emptyset$.
    To see this, note that such an $x$ would imply the existence of an empty box in $\mc{R}_I$ for some $I \le \log_2 n-2\log_2\log_2 n + d$.
    One can check that this also implies $I < \log_2\zeta^{-1}+2d$.
    However, in this range, $2^{I+3}\log_2 n / n = o(1)$ and, by \cref{cl:3.9} and a union bound, this happens with probability $O((\log_2 n)n^{-3}) = O(n^{-2})$.

    Finally, fix any integer $I \in (\log_2 n-2 \log_2\log_2 n,\, \log_2\zeta^{-1} + d]$.
    We show that, with probability $1-O(n^{-3})$, there are $O(\log_2 n)$ points $x \in \Pp$ of weight $I$ such that $\mr R[o,x]\cap \Pp=\emptyset$.
    Then, the lemma statement will follow from the discussion above and a union bound over all $I$ in the given range.
    
    Say that a box $R(i_1,\dots,i_d)$ is \emph{$I$-viable} if $R(i_1+1,\dots,i_d+1)$ is empty and $i_1+\dots+i_d = I$. Recall that if $\mr R[o,x]\cap \Pp=\emptyset$ for some $x$ of weight $I$, then $x$ lies in some viable box $R(i_1,\dots,i_d)$.
    
    We first reveal $R\cap \Pp$ for each $R\in \mc{R}_{I+d}$. By \cref{cl:3.9}, with probability $1-O(n^{-3})$, there are most $2^{I+d+3}\log_2 n/n$ empty boxes in $\mc{R}_{I+d}$.
    If this is the case, then the total volume of all viable boxes $R(i_1,\dots,i_d)$ is at most $(2^{I+d+3}\log_2 n/n)\cdot 2^{-I}=2^{d+3}\log_2 / n$. 
    By revealing which boxes are $I$-viable, we have revealed no information about the points of $\Pp$ inside those boxes. 
    Thus, conditionally on the set of $I$-viable boxes, the number of weight-$I$ points $x\in \Pp$ with $\mr R[o,x]\cap \Pp=\emptyset$, which is at most the number of points of $\Pp$ in the viable boxes, is stochastically dominated by a Poisson distribution with parameter $2^{d+3}\log_2 n$.
    By~\eqref{eq:standard}, the probability that there are more than $2^{d+4}\log_2 n$ such points is bounded from above by $2^{-2^{d+3}\log_2 n/2} \le n^{-3}$, as desired.
\end{proof}

Next, the following lemma can be viewed as a Poissonised version of \cref{lemma: few triangles from close pairs}.

\begin{lemma}\label{lemma: few triangles from close pairs in Poisson}
    Fix $d \ge 3$ and let $x \in (-1,1)^d$ be such that $\vol(\mr R[o,x]) \le \zeta$.
    Let $\Pp$ be a Poisson point process with intensity $n$ on the space $[-1,1]^d$.
    Then, with probability $1-O(n^{-3})$, the number of points $y \in \Pp$ such that $\vol(\mr R[o,y]) \le \vol(\mr R[o,x])$ and $\vol(\mr R[x,y]) \le \zeta$ is $O((\log n)^{d-2+o(1)})$.
\end{lemma}
Before proving \cref{lemma: few triangles from close pairs in Poisson}, we show how to deduce \cref{lemma: few triangles from close pairs} from it.
\begin{proof}[Proof of \cref{lemma: few triangles from close pairs}]
Recall the deduction of \cref{lemma: max degree} from \cref{lemma: degree in poisson}. In much the same way (though, union bounding over pairs of points, instead of single points), by using \cref{lemma: few triangles from close pairs in Poisson}, we see that, whp, for any two points $p,x \in P$ forming a close pair, there are only $O((\log n)^{d-2+o(1)})$ points $y \in P$ such that $pxy$ forms a triangle, $\vol(\mr R[p,y]) \le \vol(\mr R[p,x])$ and $\vol(\mr R[x,y]) \le \zeta$ is $O((\log n)^{d-2+o(1)})$.
    Also, by \cref{lemma: max degree}, whp the maximum degree of $G$ is $O((\log n)^{d-1})$.
    Now, fix any point $p \in P$. 
    Any triangle $pxy$ formed by three close edges is counted either for the pair $(p,x)$ or for the pair $(p,y)$.
    In total, there are $O((\log n)^{d-1}) \cdot O((\log n)^{d-2+o(1)}) = O((\log n)^{2d-3+o(1)})$ such triangles, as desired.
            \end{proof}

Now, we are left to prove \cref{lemma: few triangles from close pairs in Poisson}. The remaining ingredient is the following simple lemma on the volume of the region in $(-1,1)^d$ containing the points $x$ satisfying $\vol(\mr{R}[o,x])\le t$.

\begin{lemma} \label{lemma: computations for volumes}
    Fix $d \ge 1$.
    Let $\Lambda = \{x \in (-1,1)^d: \prod_k |x_k| \le t\}$ where $t \in [0,1]$.
    Then, $\vol(\Lambda) = O(t (\log \frac{1}{t})^{d-1})$.
\end{lemma}
\begin{proof}
For $i_1,\dots,i_d \ge 1$, define $R(i_1,\dots,i_d):=\{x \in \mb{R}^d: \forall k\in\{1,\dots,d\}, 2^{-i_k}\le |x_k|<2^{-i_k+1}\}$,~a~union of $2^d$ boxes. If $R(i_1,\dots,i_d) \cap \Lambda \neq \emptyset$, then $i_1+\dots+i_d \ge \log_2 t^{-1}$.
    This means that $\Lambda$ lies in the union of $R(i_1,\dots,i_d)$ such that $i_1+\dots+i_d \ge \log_2 t^{-1}$.
    So
    \begin{equation} \nonumber
        \vol(\Lambda)
        \le \sum_{I \ge \log_2 t^{-1}} \sum_{i_1+\dots+i_d=I} \vol(R(i_1,\dots,i_d))
        = \sum_{I \ge \log_2 t^{-1}} \binom{I-1}{d-1} 2^{d-I}\le \sum_{I \ge \log_2 t^{-1}} I^{d-1} 2^{d-I}.
    \end{equation}
    If $t\le 2^{-10d}$, say, then this series decays exponentially fast. Hence, 
    \[\vol(\Lambda) = O((\log_2 t^{-1})^{d-1} 2^{d-\log_2 t^{-1}})=O(t (\log \tfrac{1}{t})^{d-1}),\]
    as desired.
\end{proof}

Now, we prove \cref{lemma: few triangles from close pairs in Poisson}.
The proof comes down to estimating the area of the region where $y$ could possibly lie. 
To do this, we cover said region with $d+1$ auxiliary regions, which are easier to analyse.
For $k\in \{1,\ldots,d\}$, the $k$-th region consists of points $y$ with $|y_k-x_k|\le x_k/2$ while $\vol(\mathrm{R}[o,y])\le \zeta$. 
The last two conditions combined imply that $\prod_{i\neq k} |y_i|\le 2\zeta/x_k$ and, for every $y_k\in [x_k/2, 3x_k/2]$, the $(d-1)$-dimensional volume of the latter region is controlled by \cref{lemma: computations for volumes}.
The last region consists of points $y$ with $|y_k-x_k| > x_k/2$ for all $k\in \{1,\ldots,d\}$ while $\vol(\mathrm{R}[x,y])\le \zeta$ and $\vol(\mathrm{R}[o,y])\le \vol(\mathrm{R}[o,x])\le \zeta$.
Estimating its area is the more technical part of the proof, and is based on computations using a discrete approximation of the coordinates of $x$. 
However, there is a high-level geometric intuition why the above region should have small area: 
it is obtained from a ``star-shaped'' neighbourhood, which mostly stays close to the axis hyperplanes, by cutting away slabs centred at these hyperplanes.

\begin{proof}[Proof of \cref{lemma: few triangles from close pairs in Poisson}]
    We may and will assume that no point in $\Pp$ has a coordinate of $-1$, $0$ or $1$.
    Due to symmetry, we also assume that $x \in (0,1)^d$.
    Define $\Lambda$ to be the potential region for $y$, that is,
    $$
        \Lambda:=\left\{ y \in (-1,1)^d: \vol(\mr R[o,y]) \le \vol(\mr R[o,x]), \text{ } \vol(\mr R[x,y]) \le \zeta \right\}.
    $$
    It suffices to show that $\vol(\Lambda) = O(\zeta\cdot (\log n)^{d-2})=O((\log n)^{d-2+o(1)}/n)$: then, the lemma will follow by\footnote{This is the only place we use the assumption $d\ge 3$.} \eqref{eq:standard}. 
    To this end, consider
    $$
        \Lambda_k := \left\{ y \in (-1,1)^d: |y_k-x_k| \le x_k/2, \text{ } \vol(\mr R[o,y])\le \zeta  \right\}\quad \text{ for each } k \in \{1,\dots,d\},
    $$
    and $$
        \Lambda' := \left\{ y \in (-1,1)^d: \vol(\mr R[o,y]) \le \vol(\mr R[o,x]), \text{ } \vol(\mr R[x,y]) \le \zeta,\text{ } \forall k \in \{1,\dots,d\},\text{ } |y_k-x_k| > x_k/2 \right\}.
    $$
    It is easy to see that $\Lambda \subseteq (\bigcup_k\Lambda_k)\cup \Lambda'$ (recall that we are assuming $\vol(\mr R[o,x])\le \zeta$, so $\Lambda$ only contains $y$ with $\vol(\mr R[o,y])\le \zeta$).

    We first show that $\vol(\Lambda_k) = O(\zeta \cdot (\log n)^{d-2})$.
    Note that the set $\{y_k: |y_k-x_k|\le x_k/2\}$ is an interval of length $x_k$.
    Fix any $y_k$ in this interval and observe that $|y_k| \ge x_k/2$.
    Then, $\zeta \ge \prod_\ell |y_\ell| \ge \prod_{\ell\neq k} |y_\ell| \cdot x_k/2$, that is, $\prod_{\ell\neq k} |y_\ell| \le 2\zeta / x_k$.
    By \cref{lemma: computations for volumes}, the $(d-1)$-dimensional volume of the set of points $(y_\ell)_{\ell\ne k}\in (-1,1)^{d-1}$ satisfying this latter condition is
    \[\begin{cases}O\bigg(\frac{\zeta}{x_k} \bigg(\log \frac{x_k}{\zeta}\bigg)^{d-2}\bigg)&\text{ if }2\zeta/x_k\le 1,\\
    2^{d-1}&\text{ if }2\zeta/x_k> 1.\end{cases}\]
    If $2\zeta/x_k\le 1$, then 
    \[\vol(\Lambda_k) = x_k \cdot O\bigg(\frac{\zeta}{x_k} \bigg(\log \frac{x_k}{\zeta}\bigg)^{d-2}\bigg) = O(\zeta \cdot (\log n)^{d-2}),\]
    where we used that $x_k/\zeta\le 1/\zeta = n^{1-o(1)}$.
        On the other hand, if $2\zeta/x_k> 1$, then $x_k \le 2\zeta$, so $\vol(\Lambda_k)=2^{d-1} x_k=O(\zeta)$.

    In the rest of the proof we show that $\vol(\Lambda') = O(\zeta)$.
    For $k \in \{1,\dots,d\}$, write $i_k := \ceil{\log_2 x_k^{-1}}$, so $2^{-i_k} \le x_k < 2^{-i_k+1}$.
    Let $I:= i_1+\dots+i_d$, so $\vol(\mr R[o,x]) \le \zeta$ implies that $I \ge \log_2 \zeta^{-1}$. For $j_1,\dots,j_d\ge 1$, let $R(j_1,\dots,j_d):=\{x \in \mb{R}^d: 2^{-j_k}\le |x_k|<2^{-j_k+1}\,\, \forall k=1,\dots,d\}$.
    Similarly to the proof of \cref{lemma: computations for volumes}, we are interested in the $R(j_1,\dots,j_d)$ that intersect $\Lambda'$.
    
    Suppose $y \in R(j_1,\dots,j_d) \cap \Lambda'$.
    Using $|y_k-x_k| > x_k/2$, we know that 
    \begin{equation*}
        |y_k - x_k| = 
        \begin{cases}
            x_k - y_k\ge |y_k|, &\text{if } y_k\le x_k/2,\\
            y_k - x_k\ge y_k/3, &\text{if } y_k\ge 3x_k/2.
        \end{cases}
    \end{equation*}
        Hence, $|x_k-y_k|> \max(x_k,|y_k|)/4\ge 2^{-\min(i_k,j_k)-2}$ and
    \[
    \zeta \ge \vol(\mr R[x,y])=\prod_{k=1}^d |x_k-y_k|> 2^{-(\min(i_1,j_1)+\dots +\min(i_d,j_d))-2d}
    \]
    which implies that
    \begin{equation}\label{eq:RHS1}
     \min(i_1,j_1)+\dots +\min(i_d,j_d) > \log_2 \zeta^{-1} - 2d. 
     \end{equation}
            In addition, we know that 
    \begin{equation}\nonumber %\label{eq:RHS2}
    2^{-(j_1+\dots+j_d)} \le \vol(\mr R[o,y]) \le \vol(\mr R[o,x]) < 2^{-(i_1+\dots+i_d) + d},
    \end{equation}
    which implies that
    \begin{equation}\label{eq:RHS2'}
    j_1+\dots+j_d \ge i_1+\dots+i_d - d=I-d.    
    \end{equation}
    Let $\mc{J}$ be the set of tuples of positive integers $(j_1,\dots,j_d)$ satisfying \Cref{eq:RHS1,eq:RHS2'}.
    The discussion above indicates that $\Lambda' \subseteq \bigcup_{(j_1,\dots,j_d)\in \mc{J}} R(j_1,\dots,j_d)$.

    For every integer $J \ge I-d$, define $\mc{J}_J$ to be the set of tuples $(j_1,\dots,j_d) \in \mc{J}$ such that $j_1+\dots+j_d = J$.
    \begin{claim}\label{cl:other}
    For every $J\ge I-d$, we have
    $|\mc{J}_J|=O((J - \log_2 \zeta^{-1}+3d)^d)$.
    \end{claim}
    Note that the above bound makes sense because $J - \log_2\zeta^{-1}+3d \ge I-\log_2\zeta^{-1}+2d \ge 2d$.
    \begin{claimproof}[Proof of \cref{cl:other}]
        It suffices to show $|j_k-i_k| \le J - \log_2 \zeta^{-1}+3d$ for all $k \in \{1,\dots,d\}$.
    Suppose for the purpose of contradiction that $|j_k-i_k| > J - \log_2 \zeta^{-1} + 3d$ for some $k$.
    
    On the one hand, if $j_k > i_k + J - \log_2\zeta^{-1} + 3d$, then 
    \[\sum_{\ell\neq k} j_\ell = J-j_k < -i_k+\log_2 \zeta^{-1}-3d = \sum_{\ell\neq k}i_\ell - I + \log_2\zeta^{-1} -3d.\]
        This means $\sum_{\ell} \max(i_\ell-j_\ell,0) \ge \sum_{\ell\neq k} \max(i_\ell-j_\ell,0) > I-\log_2\zeta^{-1}+3d$, so 
    \[
        \sum_{\ell=1}^d \min(i_\ell,j_\ell)
        =\sum_{\ell=1}^d i_\ell -\sum_{\ell=1}^d \max(i_\ell-j_\ell,0)
        =I -\sum_{\ell=1}^d \max(i_\ell-j_\ell,0)<\log_2\zeta^{-1}-2d.
    \]
    This contradicts~\eqref{eq:RHS1}.
        On the other hand, if $i_k > j_k+J-\log_2\zeta^{-1}+3d$ ($ > j_k$), it holds that 
    \begin{align*}
        \sum_{\ell=1}^d \min(i_\ell,j_\ell) 
        &= \sum_{\ell=1}^d i_\ell -\sum_{\ell=1}^d \max(i_\ell-j_\ell,0)\le I - (i_k-j_k)\\
        &<I - (J-\log_2\zeta^{-1} +3d)\le I-(I-\log_2\zeta^{-1}+2d)=\log_2\zeta^{-1}-2d,    
    \end{align*}
    which also contradicts~\eqref{eq:RHS1}.
    \end{claimproof}
    
            Given the above claim, we have
    \begin{equation}\label{eq:vol}
    \vol(\Lambda') \le \sum_{J \ge I-d} |\mc{J}_J|\cdot 2^{d-J}
        = O\bigg(\sum_{J \ge I-d} (J - \log_2 \zeta^{-1}+3d)^d \cdot 2^{d-J}\bigg).    
    \end{equation}
    Note that the terms of the last sum in~\eqref{eq:vol} decay exponentially with $J$. Therefore, the entire sum is equal, up to a multiplicative factor depending only on $d$, to its first term.
            Similarly, among all $I \ge \log_2\zeta^{-1}$, the expression $(I-\log_2\zeta^{-1}+2d)^d\cdot 2^{-I}$ is dominated, up to a multiplicative factor depending only on $d$, by the value of the expression for $I = \lceil \log_2\zeta^{-1}\rceil$. By combining~\eqref{eq:vol} and the last two observations, we see $\vol(\Lambda') = O\left( \zeta \right)$, as desired.
\end{proof}

\section{A sharp upper bound on the independence number in the 2-dimensional case}\label{sec:d=2}

In this section, we prove \cref{eq:alpha-2}. As discussed in \cref{sec:outline}, this completes the proof of \cref{thm:d=2}.

The approach is as follows. 
Let $k$ be a sufficiently large multiple of $n\log \log n/\log n$, and expose the first coordinate $p_1$ of every point $(p_1,p_2)\in P$ (from here on, we will only use the randomness of the second coordinate $p_2$). Designate any $k$ of the points of $P$ (identified by their first coordinate) as being \emph{marked}, so the remaining $n-k$ points are \emph{unmarked}. By a union bound, it suffices to show that the probability that the $k$ marked points form an independent set in the Hasse diagram $H(P)$ is $o(1/\binom{n}{k})$.

To this end,
we will reveal the binary expansions of the points in $P$ bit by bit. 
At each step, having revealed $i$ bits, the points of $P$ are partitioned into $2^i$ ``levels''. We will be interested in ``consecutive pairs'' of marked points on each level, which have ``few'' unmarked points lying between them.

More precisely, the possible $2^i$ outcomes of the first $i$ bits partition the $n$ points in $P$ into $2^i$ different \emph{$i$-levels}. We say a point $(w_1,w_2)$ {\em lies between} $(p_1,p_2)$ and $(q_1,q_2)$ (at step $i$) if all three of these points are on the same $i$-level and $p_1 < w_1 < q_1$.
Two marked points $(p_1,p_2)$ and $(q_1,q_2)$ (with $p_1<q_1$) are said to be {\em consecutive after $i$ steps} if they are on the same $i$-level and no other {\em marked} point lies between them.

Set $t:=\floor{(1/2)\log_2 n}$ (this is the total number of bits that we will expose for the second coordinates of the points in $P$). We assume that $n/k$ is of the form $2^r$, and let $s := \floor{r/2}$.
For every $\ell \in \{0,1,\dots,s\}$, pick $\gamma_\ell := (1+1/r)^\ell 2^{2+r-\ell}$ and, for every $i \in \{0,1,\dots,t\}$, define $I_{\ell,i}$ to be the set of pairs $(p,q)$ such that:
\begin{itemize}
    \item $p=(p_1,p_2),q=(q_1,q_2)$ are consecutive after $i$ steps and $p_1 < q_1$, and 
    \item at most $\gamma_\ell$ points $w=(w_1,w_2)$ in $P$ lie between $p$ and $q$ at step $i$.
    \end{itemize}
Also, define the multiset \[I_\ell = \bigcup_{\ell\le i\le t/2+\ell} I_{\ell,i}.\]

The idea is that $I_\ell$ describes the set of all pairs of consecutive points whose ``distance'' is at most some threshold, over the entire duration of the bit-revelation process (here, we are using the word ``distance'' with its meaning from the sketch in \cref{subsubsec:sharper}). 
The following lemma, essentially proved by induction on $\ell$, gives us an extremely-high-probability bound on the sizes of the multisets $I_\ell$. Roughly speaking, it says, that the number of consecutive pairs at a given distance typically scales proportionally with the distance.

\begin{lemma}\label{lemma: short intervals d=2}
Let $r$ be such that $r=\Theta(\log \log n)$, and let $k=n/2^r$. As above, let $t=\floor{(1/2)\log_2 n}$ and $s = \floor{r/2}$. 
Then, with probability at least $1 - \exp(-k t/2^r)$, for every $\ell \in [0,s]$, it holds that 
    \[
        |I_\ell| \ge N_\ell
            := \bigg(1-\frac{1}{r}\bigg)^{\ell} \frac{kt}{8\cdot 2^{\ell}}\ge \frac{kt}{16\cdot 2^{\ell}}.
    \]
                \end{lemma}
\begin{proof}
    We first claim that $|I_0| \ge kt/8 = N_0$ holds deterministically.
    Indeed, for every integer $i \in [0,t/2]$, suppose we have revealed the first $i$ bits in the expansion of the second coordinates, thus determining the $i$-levels. On a level with $k'\ge 1$ marked points, there are $k'-1$ consecutive pairs of marked points, so the total number of consecutive pairs is at least $k-2^i\ge k/2$ (using that $i\le t/2\le (1/4)\log_2n$).
    Also, since $2^{r+2}\cdot k/4=n$, for at most $k/4$ consecutive pairs $(p,q)$, there are at least $2^{r+2}$ unmarked points lying between $p$ and $q$.
    So, the rest of the pairs of consecutive marked points are in $I_{0,i}$, so $|I_{0,i}| \ge k/4$.
    In total, $|I_0| \ge k/4\cdot t/2=kt/8$, as claimed.

        In the rest of the proof, we show that, for every $\ell \in [1,s]$, 
    \begin{equation} \label{eq: many intervals}
         \Pr\Big[ |I_\ell|<N_\ell\text{ and } |I_{\ell-1}|\ge N_{\ell-1} \Big] \le \exp(-kt/r^22^{\ell+6}).
    \end{equation}
    To see that this suffices, suppose that \cref{eq: many intervals} holds. Then
    \[
        \Pr\Big[|I_\ell|<N_{\ell}\text{ for some }\ell\in [0,s]\Big]
        \le \sum_{\ell=1}^s \exp(-kt/r^2 2^{\ell+6}) \le \exp(-kt/r^2 2^{s+7})\le \exp(-kt/2^r),
    \]
    as desired. (Here, we used that $s\le r/2$, $r=\omega(1)$ and $kt = \omega(2^r)$.)
        
    We are left to prove \cref{eq: many intervals}.
    To this end, we reveal the binary expansions of the second coordinates of the points in $P$, bit by bit. 
    Suppose we have revealed the first $i-1$ bits of the second coordinates of all $n$ points in $P$, for some $i\in [\ell, t/2+\ell]$.
    Fix any pair $(p,q) \in I_{\ell-1,i-1}$, and let $\mc{E}_{(p,q)}$ be the event that $(p,q) \in I_{\ell,i}$ (defined in the conditional probability space given our revealed information). 
    This event is dependent on the $i$-th bits of the second coordinates of $p,q$ and the $i$-th bits of each of the {\em unmarked} points between $p$ and $q$.
    Note that $p,q$ are on the same $i$-level with probability $1/2$ and, on this event, any unmarked point that was between $p$ and $q$ after $i-1$ steps
    remains between $p$ and $q$ 
        after $i$ steps with probability $1/2$ as well.
    Thus, a standard Chernoff bound (\cref{thm:chernoff}) implies $$
        \Pr[\mc{E}_{(p,q)}] 
        \ge \frac{1}{2} \bigg(1-\mathbb P\bigg(\mathrm{Bin}\bigg(\gamma_{\ell-1}, \frac{1}{2}\bigg) 
        > \frac{\gamma_{\ell-1}}{2} + \frac{\gamma_{\ell-1}}{2r}\bigg)\bigg)
        = \frac{1 - \exp(-\Omega(\gamma_{\ell-1}/r^2))}{2}
                \ge \frac{1-2^{-r}}{2}.
    $$
    Here, we used that $\gamma_{\ell-1}> 2^{r-\ell}\ge 2^{r/2}=\omega(r^3)$, as $r=\omega(1)$.
    
    We now claim that (conditionally on our revealed outcomes of the first $i-1$ bits of each point in $P$) the events $\mc{E}_{(p,q)}$ for $(p,q)\in I_{\ell-1,i-1}$ are mutually independent. 
    To see this, first note that we can partition the pairs $(p,q)\in I_{\ell-1,i-1}$ according to their $(i-1)$-level. 
    The points on different $(i-1)$-levels are clearly independent from each other, so we just need to show that, within each $(i-1)$-level, the corresponding $\mc E_{(p,q)}$ are mutually independent. 
    Consider any $(i-1)$-level and observe that, for two consecuitve pairs $(p,q),(p',q')$ on this level, either $p_1<q_1<p'_1<q'_1$ or $p'_1<q'_1<p_1<q_1$.
    Thus, we may order the pairs $(p,q)\in I_{\ell-1,i-1}$ on this level by the first coordinate of the left-most vertex in each pair. 
    Hence, it suffices to show that, for any $(p,q)$ on this $(i-1)$-level, if we consider all the pairs $(p',q')$ preceding $(p,q)$ (according to our first-coordinate ordering), and we condition on any joint outcome of the corresponding events $\mc{E}_{(p',q')}$, then this conditioning does not affect the probability of $\mc{E}_{(p,q)}$.
    The reason this is true is that our conditioning does not affect any of the points defining $\mc{E}_{(p,q)}$ except possibly $p$, and $\mc{E}_{(p,q)}$ is actually independent of $p$: indeed, irrespectively of the value of the $i$-th bit of $p_2$, the probability that the $i$-th bit of $q_2$ is equal to it is $1/2$.
    
    Now, consider the following thought experiment. 
    Having revealed the first $(i-1)$ digits of each point in $P$, we do \emph{not} immediately reveal the $i$-th digit of all the points in $P$ (thereby immediately revealing whether all the events $\mc{E}_{(p,q)}$ hold). 
    Instead, we consider each of the events $\mc{E}_{(p,q)}$, for $(p,q)\in I_{\ell-1,i-1}$, one by one (in some arbitrary order) and, for each such event, we check whether it holds. 
    We have proved that each of these checks succeeds with probability at least $(1-2^{-r})/2$, conditional on previous checks.

In this way, we can couple the events $\mc{E}_{(p,q)}$ (as they arise over all steps $i\in [\ell,t/2+\ell]$ of the binary revelation process) with an infinite sequence of independent Bernoulli trials (i.e., coin flips), each of which succeeds with probability $(1-2^{-r})/2$. 
Specifically, each time we check whether an event $\mc{E}_{(p,q)}$ occurs, we inspect the next Bernoulli trial in our infinite sequence and, if it succeeds, then $\mc{E}_{(p,q)}$ occurs.

If $|I_{\ell-1}|\ge N_{\ell-1}$, then over the course of the binary revelation process we inspect at least $N_{\ell-1}$ of our Bernoulli trials. If moreover $|I_\ell|<N_\ell$, then fewer than $N_\ell$ of the first $N_{\ell-1}$ Bernoulli trials that we inspect actually succeed. 
So, by a Chernoff bound (\cref{thm:chernoff} with $\mu=((1-2^{-r})/2)N_{\ell-1}$ and $\varepsilon=(r(1-2^{-r}))^{-1}$),
\begin{align*}
        \Pr\Big[ |I_\ell| < N_\ell\text{ and } |I_{\ell-1}| \ge N_{\ell-1} \Big] 
        \le&\, \Pr\left[ \mathrm{Bin}\bigg(N_{\ell-1}, \frac{1-2^{-r}}{2}\bigg) < N_\ell \right] \\
        \le&\, \Pr\left[ \mathrm{Bin}\bigg(N_{\ell-1}, \frac{1-2^{-r}}{2}\bigg) < \frac{1-2^{-r}}{2} N_{\ell-1}-\frac{N_{\ell-1}}{2r} \right] \\
        \le&\, \exp\left( \frac{1}{2r^2(1-2^{-r})^2}\cdot \frac{1-2^{-r}}{2} N_{\ell-1} \right) 
        \le\, \exp\left(-\frac{kt}{r^2 2^{\ell+6}} \right),
\end{align*}
which shows~\eqref{eq: many intervals} and finishes the proof.
\end{proof}

We now use \cref{lemma: short intervals d=2} to deduce the desired upper bound on $\alpha(H(P))$ in \cref{eq:alpha-2}.
\begin{proof}[Proof of \cref{eq:alpha-2}]
Recall that $H(P)$ is the Hasse diagram of $P$, a set of $n$ uniformly random points in $[0,1]^2$.
Set $r$ to be the largest integer satisfying $2^r\cdot r \le \log_2 n / 2^{60}$ (so $2^r=\Theta(\log n/\log \log n)$).
In addition, recall that $k = n/2^r$, and recall that we have exposed the first coordinate of every point in $P$, and identified $k$ marked points among the points in $P$. Define the event $$\cF = \{\text{the $k$ marked points span no edge in } H(P)\};$$ our objective is to show that $\Pr[\cF] \le o(1/\binom nk)$.

Note that $\binom{n}{k}\le (\e n/k)^k=(\e\cdot 2^r)^k = o(\e^{kr})$, so in fact it suffices to prove that $\Pr[\cF] \le \e^{-kr}$. This will be our objective for the rest of the proof.

Fix any $i \in [0,t/2+s]$, and suppose we have revealed the first $i$ bits in the binary expansion of the points in $P$. 
This determines the sets $I_{0,i},I_{1,i},\dots,I_{s,i}$ (which we defined immediately before \cref{lemma: short intervals d=2}); they satisfy $I_{0,i} \supseteq I_{1,i}\supseteq \dots \supseteq I_{s,i}$.

Define $w_i := \sum_{\ell=0}^s |I_{\ell,i}|2^{\ell}$; we call this the {\em score of bit $i$}.
For each pair $(p,q) \in I_{0,i}$, let $\cG_{(p,q)}$ be the event that, for some $j<r$, both $p$ and $q$ are on the same $(i+j)$-level with no unmarked points between them, but that $p$ and $q$ are on different $(i+r)$-levels, with $p_2<q_2$.
Crucially, if $\cG_{(p,q)}$ occurs, then we are able to detect that $p$ and $q$ form an edge of $H(P)$ while revealing only $r$ additional bits of the second coordinates of the points in $P$.
\begin{claim} \label{claim: probability in terms of scores}
    Conditional on our revealed outcomes of the first $i$ bits of each point in $P$, the probability that none of the events $\cG_{(p,q)}$ occur, among $(p,q) \in I_{0,i}$, is at most $\exp(-w_i/2^{r+50})$.
\end{claim}
\begin{claimproof}
    Fix any $\ell\in [0,s]$ and any pair $(p,q) \in I_{\ell,i} \setminus I_{\ell+1,i}$ (with the convention that $I_{s+1,i}=\emptyset$).
    We first show that $\Pr[\mc{G}_{(p,q)}] \ge 2^{-r+\ell-49}$.
    To this end, we first expose the next $r-\ell-1$ bits in the binary representations of the second coordinates of all points between $p$ and $q$ at step $i$.
   The points $p$ and $q$ end up in the same $(i+r-\ell-1)$-level with probability $2^{-r+\ell+1}$.
    Conditionally on this event, consider all the unmarked points that lay between $p$ and $q$ at step $i$. The probability that none of these unmarked points remain between $p$ and $q$ at step $i+r-\ell-1$ is at least 
    is at least $(1-2^{-r+\ell+1})^{\gamma_\ell} \ge (1-2^{-r+\ell+1})^{2^{3+r-\ell}}\ge \e^{-32}$.
    Here, we used that $2^{-r+\ell}\le 2^{-r/2}=o(1)$ and $1-x \ge \e^{-2x}$ when $x=o(1)$.
    If we expose one additional bit then, with probability $1/4$, the second coordinate of $q$ dominates the second coordinate of $p$, meaning that $p,q$ form an edge in $H(P)$.
        Hence, $\Pr[\mc{G}_{(p,q)}] \ge 2^{-r+\ell+1}\cdot \e^{-32}/4\ge 2^{-r+\ell-49}$, as desired.

    By similar reasoning as in the proof of \cref{lemma: short intervals d=2}, the events $\mc G_{(p,q)}$, for all $(p,q) \in I_{0,i}$, are mutually independent.
    So, using that $\e^{-x} \le 1-x$, the probability that no $\mc{G}_{(p,q)}$ occurs is at most 
    $$
      \begin{aligned}
        \prod_{\ell=0}^{s} \left( 1-2^{-r+\ell-49} \right)^{|I_{\ell,i}|-|I_{\ell+1,i}|}
        \le&\, \exp\left(-\sum_{\ell=0}^s 2^{-r+\ell-49} \left(|I_{\ell,i}|-|I_{\ell+1,i}|\right) \right) \\
        =&\, \exp\left( -2^{-r-49}|I_{0,i}| - \sum_{\ell=1}^s 2^{-r+\ell-50} |I_{\ell,i}|\right) \\
        \le &\, \exp\left( -\sum_{\ell=0}^s 2^{-r+\ell-50} |I_{\ell,i}|\right) = \exp\left(-\frac{w_i}{2^{r+50}}\right),
      \end{aligned}
    $$
    as claimed.
\end{claimproof}

We continue with the proof of \cref{eq:alpha-2}. Let $\mc E$ be the event that the conclusion of \cref{lemma: short intervals d=2} holds, i.e., that $|I_\ell|\ge kt/(16\cdot 2^\ell)$ for each $\ell\in [0,s]$. Then, $\Pr[\mc E]\ge 1-\exp(-kt/2^r)$, and when $\mc E$ holds, we have
\[\sum_{i=0}^{t/2+s} w_i \ge \sum_{\ell=0}^{s} |I_\ell|2^\ell \ge (s+1)kt/16\]
Moreover, for every $j\in \{1,\ldots,r\}$, define the event
\[
    \cE_j := \bigg\{ \sum_{i\le t/2+s,\; i\equiv j\hspace{-0.7em}\mod r} w_i \ge \frac{(s+1)kt}{16r}\ge \frac{k t}{2^5}\bigg\}
\]
and note that, by the pigeonhole principle, $\cE\subseteq \bigcup_{j=1}^r\cE_j$.

Fix any $j\in \{1,\ldots,r\}$. 
To estimate $\Pr[\mc{F} \cap \mc{E}_j]$, we reason similarly to the proof of \cref{lemma: short intervals d=2}, coupling with an infinite sequence of independent Bernoulli trials.

This time, we consider an infinite sequence of independent Bernoulli trials, which each \emph{fail} with probability $\rho:=\exp(-1/2^{r+50})$. Let $\ol{\mc G}_i$ be the event that none of the $\cG_{(p,q)}$ occur, among all $(p,q) \in I_{0,i}$. 
We can interpret $\ol{\mc G}_i$ as the event that, starting with the first $i$ bits in the binary representation of the second coordinate of each marked point, we do not manage to detect an edge of $H(P)$ by revealing the next $r$ bits.

Now, we consider an auxiliary (simplified) revelation process: imagine that we consider each $i\equiv j\hspace{-0.2em}\mod r$ in turn and, for each such $i$, we reveal $w_i$ and then we reveal whether $\ol{\mc G}_i$ occurs. By \cref{claim: probability in terms of scores}, $\ol{\mc G}_i$ occurs with (conditional) probability at most $\exp(-w_i/2^{r+50})=\rho^{w_i}$. So, we can couple our auxiliary process with our infinite sequence of Bernoulli trials so that, at the moment we check whether $\ol{\mc G}_i$ occurs, we inspect the next $w_i$ Bernoulli trials in our infinite sequence and, if they all fail\footnote{one should not necessarily ascribe any meaning to the individual Bernoulli trials; they are just a device to simulate probabilities of the form $\exp(-w_i/2^{r+50})$.} (which happens with probability $\rho^{w_i}$), then $\ol{\mc G}_i$ occurs.

If $\mc E_j$ occurs, then over the course of the binary revelation process we inspect at least $kt/2^5$ of our Bernoulli trials. If this happens but $\mc F$ does not occur, then the first $kt/2^5$ of the trials that we inspect all fail. That is to say, $\Pr[\mc{F} \cap \mc{E}_j] \le \exp(-(kt/2^{5})/2^{r+50})=\exp(-kt/2^{r+55})$.

Summing over all $j$, we obtain that 
$$
    \Pr[\mc{F}] 
    \le \Pr[\mc{E}^c] + \sum_{j=1}^{r} \Pr[\mc{F}\cap \mc{E}_j]
    \le \exp\left(-\frac{kt}{2^r}\right) + r\cdot \exp\left(-\frac{kt}{2^{r+55}}\right)
    \le n\cdot \exp\left(-\frac{kt}{2^{r+55}}\right).
$$
Now, using that $2^r\cdot r \le \log_2 n / 2^{60}$, $t = \floor{(1/2)\log_2 n}\ge \frac{1}{4}\log_2 n$ and $\exp(kr) \ge n$, we acquire $$
    \Pr[\mc{F}] \le n\cdot \exp\left(-k\cdot \frac{\log_2 n}{4}\cdot \frac{2^{60}r}{\log_2 n} 2^{-55} \right) 
    < n\cdot \exp(-2kr) \le \exp(-kr).
$$
This finishes the proof.
\end{proof}

\section{A general upper bound on the independence number}\label{sec:d>2}
In this section, we prove \cref{eq:alpha}. As described in \cref{sec:outline}, this completes the proof of \cref{thm:d>2}.

\newcommand{\dd}{r}

For a positive integer $r$ and two points $p,q\in [0,1]^\dd$, we write $p\prec q$ if all the coordinates of $p$ are less than the corresponding coordinates of $q$. Our proof of \cref{eq:alpha} mainly comes down to the following technical lemma, which we will prove by induction.

\begin{lemma} \label{lem:inductive}
Fix a positive integer $\dd$, and suppose $k,n$ are sufficiently large with respect to $\dd$, satisfying $k\le n$. Consider any $T\in[1, \sqrt{\log k}\,]$ and $Q \in [1,k^{1/\dd}/4]$, and fix a set $\Omega$ of size $Q$.
    \begin{itemize}
        \item Fix a labelling $f: \{1,\ldots,n\} \to \Omega$.
        \item Fix a set $X \subseteq \{1,\ldots,n\}$ of size $k$.
        \item Let $p_1,\dots,p_n$ be $n$ uniformly random points in $[0,1]^{\dd}$.
    \end{itemize}
We say that a point $p_i$ has \emph{level} $f(i)$, and we say that $p_i$ is \emph{marked} if $i\in X$.
Let $X^2_f$ be the set of pairs $(i,j) \in X^2$ where $p_i \prec p_j$ and $f(i)=f(j)$. For a pair $(i,j)\in X^2_f$, let 
    \[
        \on{Box}_{f}(i,j)=\{\ell\in \{1,\ldots,n\}\;:\;f(i)=f(\ell)=f(j)\text{ and } p_i \prec p_\ell \prec p_j\}
    \]
    be the set of indices of points strictly between $p_i$ and $p_j$ and on the same level. 
    
    Then, with probability at least $1-\exp(-17^{-\dd}kT)$, there is a set $\mathcal{X}\subseteq X^2_f$ of $8^{-\dd}k$ disjoint pairs $(i,j)$, each of them satisfying
    \begin{equation}\label{eq:box}
        |\on{Box}_{f}(i,j)|<\frac{8n}{k}\left(\frac{1}{8}\floor{\frac{\log k}{8\dd^{2}T\log\log k}}\right)^{1-\dd}.
    \end{equation}
    and $\on{Box}_f(i,j)\cap X=\emptyset$.
    \end{lemma}

Recalling the sketch in \cref{subsubsec: d dimensions}, we say that the pairs $(i,j)\in \mc X$ in the conclusion of \cref{lem:inductive} are ``$r$-dimensional suitable pairs''. For $r=d$, and appropriate choices of $k=n/(\log n)^{d-1+o(1)}$ and $T=\Theta(\log \log n)$, it is not hard to see that \cref{eq:box} implies that $|\on{Box}_{f}(i,j)|$ is actually empty, meaning that $p_i$ and $p_j$ form an edge of the Hasse diagram of the points $p_1,\dots,p_n$. That is to say, \cref{eq:alpha} is a fairly immediate consequence of the $r=d$ case of \cref{lem:inductive} (we will see the details of this deduction momentarily).

For the inductive step, we need to show how to find $r$-dimensional suitable pairs among random points in $[0,1]^r$, equipped with the ability to find $(r-1)$-dimensional suitable pairs among random points in $[0,1]^{r-1}$ (arbitrarily partitioned into ``levels''). We do this by first revealing the first $r-1$ coordinates of our random points (all at once), then gradually revealing their $r$-th coordinates in an $L$-ary digit-by-digit fashion (for appropriate $L$). For each $i$, knowing the first $i-1$ digits of every point naturally partitions the points into $L^{i-1}$ ``levels''. Induction provides us with $(r-1)$-dimensional suitable pairs; with the randomness of the $i$-th digit of the $r$-th coordinate of all the marked points, we study how these $(r-1)$-dimensional suitable pairs can lead to $r$-dimensional suitable pairs. (Recall from the discussion in \cref{subsubsec: d dimensions} that we only use the randomness of the $i$-th digits of the \emph{marked} points, in order to overcome certain independence issues).
\begin{proof}[Proof of \cref{eq:alpha}]Set $k=(10^5 d)^{10^5 d}n((\log\log n)^2/\log n)^{d-1}$; we will show that whp $\alpha(H(P))< k$. It is convenient to view $P$ as a set of $n$ random \emph{labelled} points $p_1,\dots,p_n\in [0,1]^d$ so, even before we expose the positions of the points, we can refer to them by their index in $\{1,\dots,n\}$.

We proceed by a union bound over all $\binom n k$ subsets  of $k$ vertices. So, fix any subset $X\subseteq \{1,\dots,n\}$ of size $k$.
In the language of \cref{lem:inductive}, with $\dd = d$ and $\Omega=\{1\}$,
we need to show that, 
with probability at least $1-o\left(1/\binom{n}{k}\right)$, there is a pair $(i,j) \in X^2_f$ such that $p_i \prec p_j$ and $\on{Box}_f(i,j)=\emptyset$.
This follows from \cref{lem:inductive} applied with $T=100^{d}\log\log n$ and the inequality $\binom n k\le (\e n/k)^k$.
\end{proof}

Now, we prove \cref{lem:inductive}.

\begin{proof}[Proof of \cref{lem:inductive}]
    We proceed by induction on $\dd$. If $\dd=1$, the desired statement
    actually holds with probability 1. Indeed, in a layer with $k'\ge 2$ marked points, we can sort the marked points in increasing order and take $\lfloor k'/2\rfloor \ge k'/2-1$ disjoint consecutive pairs (taking the first and second points, then the third and fourth points, then the fifth and sixth points, and so on). Since there are $Q$ different levels, in this way we obtain a collection $\mc Y$ of at least $k/2-Q\ge k/4$ disjoint pairs $(i,j)\in X^2$ such that the vertices in each pair lie on the same layer and satisfy $p_i \prec p_j$. 
    
    Since the pairs in $\mc Y$ are consecutive, the sets $\on{Box}_f(i,j)$, for $(i,j)\in \mc Y$, are disjoint and contain no marked points. By the pigeonhole principle, at least $k/8$ of them contain no more than $8n/k$ unmarked points, and we can take $k/8$ such sets as our desired collection $\mc X$.
    
    Now, fix $\dd\in \{2,\ldots,d\}$ and suppose that the statement is true for $\dd-1$.
    We will prove it for $\dd$. To this end, we need several definitions.
    \begin{itemize}
    \item For $i \in \{1,\ldots,n\}$, let $q_i \in [0,1]^{\dd-1}$ be the projection of $p_i$ onto the first $\dd-1$ coordinates.
        Note that $q_1,\dots,q_n$ can be viewed as $n$ uniformly random points in $[0,1]^{\dd-1}$.
    \item Let 
        \begin{equation}\label{eq: define L}
            L=\floor{\frac{\log k}{8\dd^{2}T\log\log k}} \in \left[2,\, \frac{\log k}{8\dd(\dd-1)T\log\log k} \right].
        \end{equation}
    Here, we are using that $T\le \sqrt{\log k}$ and that $k$ is sufficiently large compared to $\dd\le d$.
    \item For each $m\ge 0$, let $\delta_{m}:[0,1]^{\dd}\to\{0,\dots,L-1\}$ be the function that maps a point $p \in [0,1]^{\dd}$ to the $m$-th digit in the $L$-ary expansion of its $\dd$-th coordinate (this is uniquely defined except on a subset of $[0,1]$ of measure zero).
    \item For each $m\ge 0$, let $\Omega_{m}=\Omega\times\{0,\dots,L-1\}^{m}$ and, for $i\in \{1,\ldots,n\}$,
    let 
    \[f_{m}(i)=(f(i),\delta_{1}(p_i),\dots,\delta_{m}(p_i)).\] 
    Crucially, for each $m$, the labelling $f_m$ is independent from the projections $q_1,\dots,q_m$ (because for a uniformly random point in $[0,1]$, its last coordinate is independent from its other coordinates).
    \item Define $M=\min\{m\in\mb N:QL^m>k^{1/(\dd-1)}/4\}$. Using that $Q \le k^{1/\dd}/4$ and \cref{eq: define L}, we have
        \begin{align}\label{eq:I-L}
            M & > \frac{\log(k^{1/(\dd-1)-1/\dd})}{\log L}
            = \frac{\log k}{\dd(\dd-1)\log L} 
            \ge \frac{\log k}{\dd(\dd-1)\log \log k}
                        \ge 8LT.
        \end{align}
    \end{itemize}
    For $m<M$, note that $|\Omega_m|\le QL^m\le k^{1/(\dd-1)}/4$. For every such $m$, define $\on{Box}_{f_m}'(i,j)$ as the set of indices $\ell \in \{1,\ldots,n\}$ such that $f_m(i)=f_m(\ell)=f_m(j)$ and $q_i \prec q_\ell \prec q_j$ (this is the counterpart of $\on{Box}_f(i,j)$ with respect to the projected points $q_1,\dots,q_n$ and the labelling $f_m$).
    Let $\mathcal{E}_{m}$ be the event that there does \emph{not} exist a collection of $8^{-(\dd-1)}k$ disjoint pairs $(i,j) \in X^2$, each satisfying $q_i \prec q_j$, $f_m(i)=f_m(j)$, $\on{Box}_{f_{m}}'(i,j)\cap X=\emptyset$ and
        \[
        |\on{Box}_{f_{m}}'(i,j)|<\frac{8n}{k}\left(\frac{8}{L}\right)^{(\dd-1)-1}.
    \]
    By the inductive assumption (i.e.\ the $(\dd-1)$-case of the lemma) applied for the points $q_1,\dots,q_n$ and the labelling $f_m$, we obtain $\Pr[\mathcal{E}_{m}]\le\exp\left(-17^{-(\dd-1)}kT\right)$.
    So, using that $M = O(\log k)$, we have that
    \begin{equation}\label{eq:not-abort}
        \Pr[\mathcal{E}_{0}\cup\dots\cup\mathcal{E}_{M-1}]
        \le M\exp(-17^{-(\dd-1)}kT)\le\exp(-17^{-\dd}kT)/2.
    \end{equation}
    
    Now, consider the following recursive algorithm that attempts to find
    our desired set $\mathcal{X}$ by ``exploring consecutive digits of the $L$-ary expansion of the $\dd$-th coordinate''.
    \begin{enumerate}
        \item Let $\mathcal{X}_{0}=\emptyset$.
        \item For each $m\in\{0,\dots,M-1\}$, we define $\mathcal{X}_{m+1}$ as
        follows.
        \begin{enumerate}
            \item If $\mathcal{E}_{m}$ occurs, abort the entire algorithm. 
                Otherwise, let $\mathcal{Y}_{m}$ be a collection of $8^{-(\dd-1)}k$ disjoint pairs $(i,j) \in X^2$, each satisfying $q_i \prec q_j$, $f_m(i)=f_m(j)$ and
                                \[
                    |\on{Box}_{f_{m}}'(i,j)|<\frac{8n}{k}\left(\frac{8}{L}\right)^{(\dd-1)-1}.
                \]
            \item Let $\mathcal{Y}_{m}'$ be obtained from $\mathcal{Y}_{m}$ by deleting
            every pair which intersects a pair in $\mathcal{X}_{m}$. 
            In~particular, if $|\mathcal{X}_{m}|<8^{-\dd}k$, then 
            $|\mathcal{Y}_{m}'|\ge |\mathcal{Y}_{m}|-2|\mathcal{X}_{m}|\ge 8^{-\dd}k$.
            \item For every $(i,j)\in\mathcal{Y}_{m}'$, let $S_{i,j}\subseteq\{0,\dots,L-1\}$
            be the set of all $s\in\{0,\dots,L-1\}$ such that
                \[
                \Big|\Big\{\ell\in \on{Box}_{f_{m}}'(i,j):\delta_{m+1}(p_\ell)=s\Big\}\Big|
                    \le\frac{4}{L}|\on{Box}_{f_{m}}'(i,j)|.
                \]
            Hence, $|S_{i,j}|\ge 3L/4$.
            \item Let $S_{i,j}'$ be the set of $s \in \{0,\dots,L-2\}$ such that $\{s,s+1\} \subseteq S_{i,j}$.
                We know that $$|S_{i,j}'| \ge |S_{i,j}|-1-(L-|S_{i,j}|) \ge L/4.$$
            \item We obtain $\mathcal{X}_{m+1}$ from $\mathcal{X}_{m}$ by adding
            all pairs $(i,j)\in\mathcal{Y}_{m}'$ for which $\delta_{m+1}(i)=\delta_{m+1}(j)-1\in S_{i,j}'$.
                For each such pair, we are guaranteed to have $p_i \prec p_j$ and 
                \[
                    |\on{Box}_{f}(i,j)|\le2\cdot \frac{4}L|\on{Box}_{f_{m}}(i,j)|
                    <\frac{8n}{k}\left(\frac{8}{L}\right)^{\dd-1}.
                \]
        \end{enumerate}
        \item In the end, we set $\mc{X} := \mc{X}_{M}$.
    \end{enumerate}
    Note that, if the algorithm does not abort, then it produces a set $\mathcal{X}$ satisfying the properties required in the lemma except that $|\mathcal{X}|$ might not be large enough (recall that we need $|\mathcal{X}_{M}|\ge 8^{-\dd}k$). 
    Recalling \cref{eq:not-abort}, we already know that the algorithm is unlikely to abort, so we artificially set $|\mathcal{X}|=\infty$ when it does. For the rest of the proof, our objective is to show that
    \[
        \Pr[|\mathcal{X}|<8^{-\dd}k]\le\exp\left(-17^{-\dd}kT\right)/2.
    \]
    Now, we explain how to couple the above algorithm with a sequence of Bernoulli trials in such a way that, whenever $|\mathcal{X}|<8^{-\dd}k$, an atypically small number of our Bernoulli trials must have succeeded
    (which allows us to conclude by applying a standard Chernoff bound).
    
    To this end, first note that, for each $m \in \{0,\dots, M-1\}$, to determine the set $\mathcal{Y}_{m}'$,
    we only need to reveal $q_1,\dots,q_n$ and $f_{m}(1),\dots,f_m(n)$ (that is,
    we reveal the first $\dd-1$ coordinates of each point and the first $m$ digits in the $L$-ary expansion of the last coordinate of each point). 
    As mentioned in the description of the algorithm, if we do not already have $|\mathcal{X}_{m}|\ge 8^{-\dd}k$ (which would directly yield $|\mathcal{X}|\ge 8^{-\dd}k$), then we have $|\mathcal{Y}_{m}'|\ge 8^{-\dd}k$. 
    Next, to determine the sets $S_{i,j}$ and $S_{i,j}'$, we need to further
    reveal $\delta_{m+1}(p_i)$ for $i \in \{1,\ldots,n\}\setminus X$, (that is, reveal the $(i+1)$-th
    digit in the $L$-ary expansion of the last coordinate of each unmarked point).
    Conditionally on the revealed information, the values of $\delta_{m+1}(p_i)$ for $i\in X$ are still uniformly random, so each of the (at least $8^{-\dd}k$) pairs $(i,j)\in\mathcal{Y}_{m}'$ contributes to $\mathcal{X}_{m+1}$ with probability at least $(|S_{i,j}'|/L)\cdot(1/L)\ge 1/(4L)\ge 2T/M$
    (where the last inequality uses \cref{eq:I-L}).
    
    The above considerations hold for all $m\in\{0,\dots,M-1\}$, so the event $|\mathcal{X}|<8^{-\dd}k$ can be coupled with a subset of the event that, among $M\cdot8^{-\dd}k$ independent $\on{Bernoulli}(2T/M)$
    trials, at most $8^{-\dd}k$ succeed. 
    By a Chernoff bound (e.g. \cref{thm:chernoff} with $\mu=2\cdot8^{-\dd}kT$ and $\varepsilon=1/2$) this occurs with probability at most
    \[
        \exp\left(-2\cdot 8^{-\dd}kT / 8\right) = \exp\left(-\frac{8^{-\dd}kT}{4}\right)\le\exp\left(-17^{-\dd}kT\right)/2,
    \]
    which finishes the induction step and the proof.
\end{proof}

\begin{remark} \label{remark: improvable}
    In our proof of \cref{lem:inductive}, it was convenient to consider only \emph{disjoint} pairs. Actually, replacing our elementary calculations with a more sophisticated concentration inequality (specifically, \emph{Suen's inequality}, see for example \cite[Theorem~2.22]{RandomGraph00}), we can tolerate some intersections between our pairs, which allows us to consider a larger family of pairs at each step, and therefore slightly improve the bounds in \cref{lem:inductive}.
    Specifically, for $\dd\ge 3$ we can maintain $\Omega(k\log\log n)$ pairs $(i,j) \in X^2_f$ such that $|\on{Box}_{f}(i,j)|$ is small, with the property that 
        every $i$ participates in $O(\log\log n)$ such pairs. With this idea (and some other minor adjustments to the proof), one can show that
    whp $\alpha(H(P))$ is $O(n(\log\log n)^d/(\log n)^{d-1})$ (i.e., one can improve on \cref{eq:alpha} by a factor of $(\log\log n)^{d-2}$). 
    Actually, we can improve this by one additional $\log \log n$ factor: while it is not clear how to incorporate the ideas in \cref{lemma: short intervals d=2} in \emph{all} of the inductive steps, we can at least incorporate these ideas in the $\dd=2$ step of the induction, which would improve the upper bound to $O(n(\log\log n/\log n)^{d-1})$.
    
    Since we believe that this is still sub-optimal for all $d\ge 3$, we omit the details of these improvements.
\end{remark}

\paragraph{\textbf{Acknowledgements.}} 
Jin was supported by SNSF grant 200021-228014.
Kwan was supported by ERC Starting Grant ``RANDSTRUCT'' No. 10107677. Lichev was supported by the Austrian Science Fund (FWF) grant No. 10.55776/ESP624.

\bibliographystyle{plain}
\bibliography{Bib}

\end{document}